\documentclass[11pt,a4paper]{article}
\usepackage{amsmath}
\usepackage{amsfonts}
\usepackage{amssymb}
\usepackage{amsmath}
\usepackage{amsthm}
\usepackage{latexsym}
\usepackage{hyperref}
\usepackage{graphicx}
\usepackage{fullpage}
\usepackage{epstopdf}
\usepackage{float}
\usepackage{color}
\usepackage{accents}
\def\eps{\varepsilon}

\newcommand{\Sw}{\mathcal{S}}
\newcommand{\laps}[1]{(-\lap)^{\frac{#1}{2}}}
\newcommand{\lapms}[1]{I_{#1}}

\newcommand{\lap}{\Delta}

\newcommand{\Rz}{\mathcal{R}}

\newcommand{\loc}{\mathrm{loc}}

\newcommand{\aleq}{\prec}

\newcommand{\dist}{\mathrm{dist}}

\newcommand{\R}{\mathbb{R}}
\newcommand{\N}{\mathbb{N}}

\newcommand{\vrac}[1]{\| #1 \|}

\newtheorem{theorem}{Theorem}[section]

\newtheorem{definition}[theorem]{Definition}

\newtheorem{lemma}[theorem]{Lemma}

\newtheorem{proposition}[theorem]{Proposition}
\newtheorem{remark}[theorem]{Remark}

\date{}

\begin{document}
\title{Blow-up behaviour of a fractional Adams-Moser-Trudinger type inequality in odd dimension}
\author{Ali Maalaoui, Luca Martinazzi\thanks{L.M., A.M. supported by Swiss National Science Foundation.}, Armin Schikorra\thanks{A.S.' research leading to these results has received funding from the European Research Council under the European Union's Seventh Framework Programme (FP7/2007-2013) / ERC grant agreement n° 267087. A.S. is partially supported by Swiss National Science Foundation.}}
\maketitle

\begin{abstract}
Given a smoothly bounded domain $\Omega\Subset\R^n$ with $n\ge 1$ odd, we study the blow-up of bounded sequences $(u_k)\subset H^\frac{n}{2}_{00}(\Omega)$ of
solutions to 
the non-local equation
$$(-\Delta)^\frac n2 u_k=\lambda_k u_ke^{\frac n2 u_k^2}\quad \text{in }\Omega,$$
where $\lambda_k\to\lambda_\infty \in [0,\infty)$, and
% critical points to the fractional Moser-Trudinger-type inequality
%$$\sup_{u\in H^{n/2}_{00}(\Omega), \, \|(-\Delta)^{n/4} u\|_{L^2(\R^n)}^2\le \Lambda_1} \int_{\Omega} e^{\frac{n}{2}u^2}dx <\infty,\quad \Lambda_1:=(n-1)! \mathrm{vol}(S^n)$$
$H^{\frac n2}_{00}(\Omega)$ denotes the Lions-Magenes spaces of functions $u\in L^2(\R^n)$ which are supported in $\Omega$ and with $(-\Delta)^\frac{n}{4}u\in L^2(\R^n)$.
Extending previous works of Druet, Robert-Struwe and the second author, we show that if the sequence $(u_k)$ is not bounded in $L^\infty(\Omega)$, a suitably rescaled subsequence $\eta_k$ converges to the function $\eta_0(x)=\log\left(\frac{2}{1+|x|^2}\right)$, which solves the prescribed non-local $Q$-curvature equation
$$(-\Delta)^\frac n2 \eta =(n-1)!e^{n\eta}\quad \text{in }\R^n$$
recently studied by Da Lio-Martinazzi-Rivi\`ere when $n=1$, Jin-Maalaoui-Martinazzi-Xiong when $n=3$, and Hyder when $n\ge 5$ is odd.
We infer that blow-up can occur only if $\Lambda:=\limsup_{k\to \infty}\|(-\Delta)^\frac n4 u_k\|_{L^2}^2\ge \Lambda_1:= (n-1)!|S^n|$.% Moreover if $\Lambda=\Lambda_1$, and blow up occurs, then $u_k\rightharpoonup 0$ weakly in $H^\frac{n}{2}_{00}(\Omega)$.
\end{abstract}
\section{Introduction}

In this paper we study some compactness properties of the embedding of $H^{\frac n2}_{00}(\Omega)$ into Orlicz spaces, where $\Omega$ is a smoothly bounded domain in $\R^n$. In order to introduce the relevant function spaces we start by recalling various definitions of fractional Laplacians.

Let $\mathcal{S}(\R^n)$ denote Schwarz space of smooth and rapidly decreasing functions on $\R^n$. For a function $u\in\mathcal{S}(\R^n)$ and for $s \in (0,\infty)$,
we define
\[
 \laps{s}u := (|\cdot|^{2s} u^\wedge)^\vee.
\]
Here the Fourier transform is defined via
\[
  u^\wedge(\xi) \equiv \mathcal{F}u(\xi) :=\frac{1}{ (2\pi)^{n/2}} \int_{\R^n} e^{-i x\cdot \xi}\ u(x) dx.
\]
and $u^\vee$ is its inverse.

For $s \in (0,2)$ one can also prove (see e.g. \cite{DNPV}) that for a certain constant $c_{n,s} \in \R$
\[
 \laps{s}u(x) = c_{n,s} P.V. \int_{\R^n} \frac{u(x+h)-u(x)}{|h|^{n+s}}\, dh. 
\]
In order to define the operator $(-\Delta)^s$ on a space larger than the Schwarz space, set for $s>0$
\begin{equation}\label{defLs} 
L_s(\R^n):=\left\{u\in L^1_{\loc}(\R^n): \int_{\R^n}\frac{|u(x)|}{1+|x|^{n+s}}dx<\infty\right\}.
\end{equation}
Then for $u\in L_s(\R^n)$ we can define
$(-\Delta)^s u$ as a tempered distribution as follows:
$$\langle (-\Delta)^\frac{s}{2} u,\varphi\rangle:=\int_{\R^n} u(-\Delta)^\frac{s}{2} \varphi dx,\quad \text{for }\varphi\in\mathcal{S}(\R^n).$$
This is due to the fact that for $\varphi\in \mathcal{S}(\R^n)$ one has $(1+|x|^{n+s})|(-\Delta)^\frac{s}{2} \varphi(x)|\le C$ for a constant depending on $\varphi$ but not on $x$, see \cite[Proposition 2.2]{Ali} and \cite{Sil}. % \armin{shouldn't we for politeness cite Silvestre's work here?}.

We can now define the space
$$H^s(\R^{n}):=\{u\in L^2(\R^n):(-\Delta)^\frac s2 u\in L^2(\R^n)\},$$
endowed with the norm  
$$\|u\|_{H^s(\R^n)}^{2}:=\|u\|_{L^2(\R^n)}^2+\|(-\Delta)^\frac s2 u\|_{L^2(\R^n)}^2,$$
where with the expression $(-\Delta)^\frac s2 u\in L^2(\R^n)$ we mean that the tempered distribution $(-\Delta)^\frac s2 u$ can be represented by a square-summable function. 

%Other definitions of $H^s(\R^n)$ are possible. For instance we mention that when $s\in (0,1)$ then a norm equivalent to $\|\cdot \|_{H^s(\R^n)}$ is
%\[
%\|u\|^2_{W^{s,2}(\R^n)} :=\|u\|_{L^2(\R^n)}^2+\int_{\R^n} \int_{\R^n} \frac{|u(x+h)+u(x)|^2}{|h|^{n+2s}}\ dh\ dx,
%\]
%see e.g. \cite{DNPV}.

Given a bounded set $\Omega\Subset \R^n$ we also define its subspace
$$H_{00}^{s}(\Omega):=\{u\in H^s(\mathbb{R}^{n}) : u\equiv0 \text{ on } \Omega^{c}\}.$$
In particular we will consider the space $X(\Omega):= H^{\frac n2}_{00}(\Omega)$ for $n$ is odd,
endowed with the norm
$$\|u\|_{X}^{2}:=\|(-\lap)^\frac n4 u\|_{L^2(\R^n)}^2=\int_{\mathbb{R}^{n}}|\xi|^{n}|\hat{u}(\xi)|^{2}d\xi.$$
The norms $\|\cdot \|_X$ and $\|\cdot\|_{H^\frac n2(\R^n)}$ are equivalent on $H_{00}^\frac n2(\Omega)$ by a Poincar\'e-type inequality. % This follows for instance from Theorem 7.1 in \cite{gru}. ( \armin{with the constant depending on $|\Omega|$, prove via Poincare or cite?}. \luca{We can discuss it.})
The space $H_{00}^s(\Omega)$ is also known as Lions-Magenes space, and is sometimes denoted by
$\tilde H^s(\Omega)$, or even $L^{s,2}_0(\Omega)$.

\medskip

We recall the following fractional version of the Adams-Moser-Trudinger inequality, see \cite[Theorem 1]{MS}:
\begin{theorem}\label{trmMT}
For any integer $n>0$ there exists a constant $C_n>0$ such that for every open set $\Omega\subset\R^n$ with finite volume $|\Omega|$ one has
\begin{equation}\label{probsup}
\sup_{u\in X(\Omega),\,\|u\|_{X}^2\le \Lambda_1}\int_{\Omega} e^{\frac n2 u^{2}} dx\le C_n|\Omega|,
\end{equation}
where $\Lambda_1:=(n-1)!|S^n|$.
\end{theorem}

When $n=2$ the above theorem is a special case of the Moser-Trudinger inequality \cite{tru}, and when $n>2$ is even it is a special case of Adams' inequality \cite{ada}.

\medskip

In this paper we want  to study the blow-up behavior of extremals of \eqref{probsup}, i.e.
weak solutions $u\in X(\Omega)$ of the  Euler-Lagrange equation  
\begin{equation}\label{equaprob}
(-\Delta)^{\frac{n}{2}}u=\lambda ue^{\frac n2 u^{2}},\quad \text{for some }\lambda\in \R,
\end{equation}
which can be intended in the following sense:

\begin{definition}
Given $f\in L^{2}(\Omega)$, a function with $u\in X(\Omega)+\R$ (i.e. $u+c\in X(\Omega)$ for some $c\in\R$) is a weak solution to 
\begin{equation}\label{eq}
(-\Delta)^{\frac{n}{2}}u=f\quad \text{in $\Omega$}
\end{equation}
if
\begin{equation}\label{eqweak}
\int_{\R^n} (-\lap)^\frac n4 u\, (-\lap)^\frac n4 \varphi dx = \int_{\R^n} f \varphi dx, \quad \forall \varphi\in X(\Omega).
\end{equation}
\end{definition}

\begin{remark}\label{rmkinteg} It follows from Theorem \ref{trmMT} that for $u\in X(\Omega)$ one has $e^{u^2}\in L^p(\Omega)$ for every $p\in [1,\infty)$ (see also \cite[Theorem 9.1]{pee}). In particular the right-hand side of \eqref{equaprob} belongs to $L^p(\Omega)$ for $p\in [1,\infty)$.
\end{remark}

The Lagrange multiplier $\lambda$ in \eqref{equaprob} can be computed by testing the equation with $\varphi=u$ (in the spirit of \eqref{eqweak}). This leads to 
\begin{equation}\label{energy}
\|(-\lap)^{\frac{n}{4}} u\|_{L^2(\R^n)}^2=\lambda\int_\Omega u^2e^{\frac n2 u^2}dx,
\end{equation}
whence $\lambda>0$, unless $u\equiv 0$.
% We also define the space 
% $$H^{\frac{1}{2}}_{00}(\Omega):=\{u\in H^{1/2}(\mathbb{R}^{3}); u=0 \text{ on } \Omega^{c}\},$$
% endowed with the norm 
% $$\|u\|_{H^{1/2}}^{2}:=\int_{\mathbb{R}^{3}}|\xi||\hat{u}|^{2}d\xi =\int_{\mathbb{R}^{3}}|(-\Delta)^{\frac{1}{4}}u|^{2}dx.$$

We are interested in the study of the blowing-up behavior of a sequence of continuous solution to the following problem :
\begin{equation}\label{eq0}
\left\{\begin{array}{ll}
(-\Delta)^{\frac{n}{2}}u_{k}=\lambda_{k}u_{k}e^{\frac{n}{2}u_{k}^{2}} &\text{in $\Omega$}\\
%u_k > 0 &\text{in $\Omega$}\\
u_{k} \in X(\Omega) &
\end{array}
\right.
\end{equation}
where $\lambda_{k}\geq 0$. 

\begin{remark}\label{rmkreg} It follows from Remark \ref{rmkinteg}, from the estimates in \cite{gru} and bootstrapping, that every solution $u$ to \eqref{equaprob} belongs to $C^{\frac{n-1}{2},\alpha}(\bar\Omega)\cap C^\infty(\Omega)$ for some $\alpha\in (0,1)$, and in fact  the function $d^{-\frac n2}u:\Omega\to\R$, where $d$ is the distance function from $\partial\Omega$, can be extended to a function in $C^\infty(\bar\Omega)$. In particular, $\sup_{\Omega} u_k \in \R$.
\end{remark}

The main result of this paper can be stated as follows :

\begin{theorem}\label{main}
Consider a \emph{bounded} sequence $(u_{k})_{k\in \mathbb{N}} \subset X(\Omega)$ of solutions to (\ref{eq0}). Set $m_{k}:=\sup_{\Omega} |u_{k}|$ and 
$$\Lambda:=\limsup_{k\to\infty}\|u_k\|_{X}^2<\infty.$$ Up to possibly replacing $u_k$ with $-u_k$ we can assume that $m_k=\sup_{\Omega} u_k$ for every $k$. Assume also that $0<\lambda_k\le \bar \lambda$ for some $\bar \lambda<\infty$ and that $\lambda_k\to\lambda_\infty$ as $k\to\infty$. Then up to extracting a subsequence one of the following holds:\\
\begin{itemize}
\item[(i)] $\lim_{k\to\infty} m_{k}<\infty$ and $u_{k}$ converges to $u_{\infty}$ in $C^{\ell}_{\loc}( \Omega)\cap C^\frac{n-1}{2}(\bar\Omega)$  for any $\ell\in\mathbb{N}$ and $u_{\infty}\in C^\ell_{\loc}(\Omega)\cap C^\frac{n-1}{2}(\bar\Omega)\cap X(\Omega)$ solves
$$(-\Delta)^{\frac{n}{2}}u_\infty=\lambda_\infty u_\infty e^{\frac{n}{2}u_\infty^{2}} \quad \text{in }\Omega.$$

\item[(ii)]  $\lim_{k\to\infty} m_{k}=\infty$, $\Lambda\ge \Lambda_1$, with $\Lambda_1$ as in Theorem \ref{trmMT}, % \luca{$u_k\rightharpoonup u_\infty$ in $X(\Omega)$ and
%\begin{equation}\label{energylim}
%\lim_{k\to\infty} \|u_{k}\|_{X}^2\geq \|u_\infty\|_X^2+ \Lambda_{1}:=(n-1)!|S^n|.
%\end{equation}
%}
and setting $r_{k}$ such that
\begin{equation}\label{eq:rkdef}
\lambda_{k}m_{k}^{2}e^{\frac{n}{2} m_{k}^{2}}r_{k}^{n}=2^n(n-1)!,
\end{equation} 
and
\begin{equation}\label{defetak}
\eta_{k}(x):=m_{k}(u_{k}(x_{k}+r_{k}x)-m_{k}),\quad \eta_0(x):=\log\left(\frac{2}{1+|x|^{2}}\right),
\end{equation}
 one has $\eta_k +\log 2 \to\eta_0$ in $C^{\ell}_{\loc}(\R^n)$ for every $\ell\ge 0$.
\end{itemize}
%In particular in $0\le \Lambda<\Lambda_1$, then only case (i) can occur.
\end{theorem}

Since Theorem \ref{main} was proven in \cite{AS}, \cite{dru}, \cite{RS} and \cite{mar4} when $n$ is even, we shall only consider the remaining case $n$ odd. In some proofs we will focus on the case $n\ge 3$, but simple modifications make every argument work for the case $n=1$. In fact the case $n=1$ is a slightly simpler, since comparison principles and in particular the Harnack inequality are available. 

The general strategy of the proof is similar to the one in the even-dimensional case, but some new difficulties arise due to the nonlocal nature of the operator $(-\Delta)^\frac{n}{2}$, as we shall now describe.

One would like to shows that in case of blow-up (Case (ii) in Theorem \ref{main}) the functions $\eta_k$ converge to a function $\eta_0\in L_n(\R^n)$ solving
\begin{equation}\label{eqeta0}
(-\Delta)^\frac n2 \eta_0=(n-1)! e^{n\eta_0}\quad \text{in }\R^n,\quad V:=\int_{\R^n} e^{n\eta_0}dx<\infty,
\end{equation}
and then prove that, among all solutions to \eqref{eqeta0}, $\eta_0$ has the special form given by \eqref{defetak}.

The first problem is that the local convergence of $\eta_k$ to a function $\eta_0$ rests on local gradient bounds for $\eta_k$ not depending on $k$ (when $n=1,2$ such bound are not necessary, thanks to the Harnack inequality). This is the content of Propositions \ref{la:ulapuest} and \ref{gradest}, one of the crucial parts of the paper.  In particular we will show that for $s < n$,
\begin{equation}\label{stimagu}
\int_{B_\rho(x_0)}|u_k(-\Delta)^\frac s2 u_k|\, dx\le C \rho^{n-s},\quad \text{for }B_{10\rho}(x_0)\subset\Omega.
\end{equation}
In the previous work \cite{mar4} an analogous estimate was obtained by noticing that $(-\Delta)^\frac{n}{2}(u_k^2)$ is uniformly bounded in $L^1(\Omega)$ when $n$ is even. Unfortunately this was based on an explicit expansion of $(-\Delta)^\frac{n}{2} (u_k^2)$ as sum of partial derivatives of $u_k$, which is of course not possible when $n$ is odd. Here instead we reduce \eqref{stimagu} to a the bound
$$\|u_k(-\Delta)^\frac s2 u_k\|_{L^{(\frac ns,\infty)}(B_\rho(x_0))}\le C,$$
which will be proven writing $u_k(-\Delta)^\frac s2 u_k$ in terms of the Riesz potential. The formal heuristic argument goes as follows. Write formally
\begin{equation}\label{formalid}
\begin{split}
u_k(-\Delta)^\frac s2 u_k &=u_k I_{n-s}(-\Delta)^\frac{n}{2} u_k\\
&=: (I_{\frac n2}(-\Delta)^\frac{n}{4}u_k)I_{n-s}(\theta(-\Delta)^\frac{n}{2} u_k) + I_{n-s}(u_k(-\Delta)^\frac{n}{2} u_k)+E\\
&=: A+B+E,
\end{split}
\end{equation}
where $\theta\in C^\infty_0(B_{2\rho}(x_0))$ is a cut-off functions, $I_t$ denotes the Riesz potential, and $E$ is an error term, which can be bounded using a commutator-type estimate.
%\armin{I don't think this describes it correctly. What I inherently used is \eqref{eq:estIII} (this is why I chose this argument). So it should say:
%\[
% u_k(-\Delta)^\frac s2 u_k = I_{n-s}(u_k(-\Delta)^\frac{n}{2} u_k) + ...
%\]
%The first term is bounedd by the PDE, and the ... are also bounded by other arguments. 
%}
Then one has to bound the term $A$ in $L^{(\frac{n}{s},\infty)}(B_\rho(x_0))$ using that $(-\Delta)^\frac{n}{4} u_k$ is bounded in $L^2(\R^n)$, while $(-\Delta)^\frac{n}{2} u_k$ is bounded in $L\log^\frac12 L(\Omega)$. These are borderline estimates, for instance because $I_{\frac{n}{2}}$ fails to send $L^2$ into $L^\infty$. Using elementary tricks we are able to circumvent this problem, obtaining Propositions \ref{borderline1}  and \ref{borderline2}. In order to bound $B$ one uses the PDE, and in particular that $u_k(-\Delta)^\frac n2 u_k$ is bounded in $L^1(\Omega)$.
Finally, to move from the formal argument to a rigorous one, and in particular to replace the first identity in \eqref{formalid} with a correct identity, we have to approximate $u_k$ with functions in $C^\infty_c(\R^n)$. The necessary technical results are contained in Section \ref{sec:borderline} and the appendix. 

The second problem, still related to the non-local nature of $(-\Delta)^\frac n2$, is that uniform estimates on the derivatives of the blown-up functions $\eta_k$ do indeed guarantee that $\eta_k\to\eta_0$ in $C^\ell_{\loc}(\R^n)$ (up to the additive constant $\log 2$ which we shall now ignore) for $\ell\le n-1$, but why should the convergence
\begin{equation}\label{convetak}
(-\Delta)^\frac{n}{2} \eta_k \to (-\Delta)^\frac{n}{2} \eta_0 \quad\text{in }\mathcal{S}'(\R^n) \text{ as }k\to\infty
\end{equation}
hold? Indeed \eqref{convetak} means that
\begin{equation}\label{convetak2}
\lim_{k\to\infty}\int_{\R^n} \eta_k \,(-\Delta)^\frac{n}{2} \varphi dx \to \int_{\R^n} \eta_0 \,(-\Delta)^\frac{n}{2} \varphi dx \quad \text{for every }\varphi \in \mathcal{S}(\R^n),
\end{equation}
and since even for $\varphi\in C^\infty_c(\R^n)$ we have that $(-\Delta)^\frac{n}{2}\varphi$ is not compactly supported, the local convergence of $\eta_k$ is not sufficient to guarantee \eqref{convetak2}. A priori it is not to be ruled out that while $\eta_k\to \eta_0$ very nicely in a compact set, ``at infinity'' $\eta_k$ has a wild behaviour. To rule this out we shall prove uniform bounds of $\eta_k$ in $L_s(\R^n)$ for any $s>0$, which is the content of Proposition \ref{boundLe}. Here we critically use that $\eta_k$ is uniformly upper bounded by construction, and the local bounds on the derivatives of $\eta_k$.

At this point it will be easy to conclude that $\eta_k$ locally converges to a function $\eta_0\in L_n(\R^n)$ solving \eqref{eqeta0}. Now we are faced with the problem of determining $\eta_0$. Indeed, similarly to what was shown in \cite{CC}, also in odd dimension $3$ or higher, Problem \eqref{eqeta0} has many solutions, as shown in \cite{JMMX} (when $n=3$) and \cite{Ali0} (for any $n\ge 3$ odd). Here we are able to use the following recent result of Ali Hyder, together with the previous bounds to show that among all solutions of \eqref{eqeta0} actually $\eta_0$ is a special one, precisely the one given in \eqref{defetak}.

\begin{theorem}[A. Hyder \cite{Ali}]\label{thmali}
Let $\eta_0\in L_n(\R^n)$ solve \eqref{eqeta0}. Then $\eta_0$ can be decomposed as $\eta_0= v+P$, where $P$ is a polynomial of degree at most $n-1$, and $v(x)=-\alpha \log(|x|)+o(\log|x|)$ as $|x|\to \infty$. Moreover $P$ is constant if and only if
\begin{equation}\label{spherical}
\eta_0(x)=\log\frac{2\lambda}{1+\lambda^2|x-x_0|^2},\quad \text{for some }\lambda>0,\quad x_0\in\R^n.
\end{equation}
%If $P$ is not constant, then there exists $j$ with $1\le j\le \frac{n-1}{2}$ and a constant $c\ne 0$ such that
%\begin{equation}\label{conda}
%\lim_{|x|\to\infty}\Delta^j \eta_0= c.
%\end{equation}
\end{theorem}

Indeed, if $\eta_0$ is not of the form \eqref{spherical}, then $\eta_0$ at infinity behaves like a logarithm plus a polynomial, only the former belonging to $L_s(\R^n)$ for $s$ small. This is in contradiction to the fact that $\eta_0\in L_s(\R^n)$ for every $s>0$. This argument is different from the one used in the even dimensional case, first introduced in \cite{RS} and then also applied in \cite{mar4} and other works.

In the case $n=1$ Theorem \ref{thmali} is not necessary because Da~Lio-Martinazzi-Rivi\`ere \cite{DLMR} proved that every function  $\eta_0\in L_1(\R)$ solving \eqref{eqeta0} for $n=1$ has necessarily the form \eqref{spherical}.

\medskip

It has to be mentioned that in even dimension the analog of Theorem \ref{main} was complemented in \cite{dru}, \cite{MS2} and \cite{str2} by a quantization result, saying that in case of blow-up
$$\Lambda=\int_\Omega |\nabla^\frac n2 u_\infty|^2 dx + L\Lambda_1\quad \text{for an integer }L>0.$$
In other words the energy loss in the weak limit is an integer multiple of the fixed quantity $\Lambda_1$. Although it is natural to expect this to hold true also in our non-local case, we remark that in the local case the proofs make abundant use of ODE techniques, which are not available when dealing with fractional Laplacians. On the other hand in the case of half-harmonic maps, precise energy quantization was obtained in \cite{DaLioBubbles}.

\subsection*{Notation}
The space $C^\alpha(\Omega) \equiv C^{\alpha_0,\tilde{\alpha}}(\Omega)$,  for $\alpha = \alpha_0 + \tilde{\alpha}$ with $\tilde{\alpha} \in (0,1]$, $\alpha_0 \in \N_0$, is the space of $\alpha_0$-times differentiable functions with $\alpha_0$th derivative H\"older continous of order $\tilde{\alpha}$.
We define the semi-norm
\[
 [f]_{C^{\alpha}(\Omega)}= \sup_{x\neq y \in \Omega} \frac{|\nabla^{\alpha_0} f(x) - \nabla^{\alpha_0}f(y)|}{|x-y|^{\tilde{\alpha}}},
\]
and the norm
\[
 \|f\|_{C^{\alpha}(\Omega)} := \sum_{k=0}^{\alpha_0} \|\nabla^{k}  f\|_{L^{\infty}(\Omega)} + [f]_{C^{\alpha}(\Omega)}.
\]

\section{Proof of Theorem \ref{main}}

%\luca{
%\begin{lemma}\label{lambdak} For every $k$ we have $\lambda_k\le \lambda_1(\Omega)$, where $\lambda_1(\Omega)$ denotes the first eigenvalue of $(-\Delta)^\frac{n}{2}$ in $\Omega$,

%\end{lemma}

%}

\begin{proposition} If $\sup_k m_k\le C$ then up to a subsequence $u_k\to u_\infty$ in $C^{\ell}_{\loc}(\Omega)\cap C^{\frac{n-1}2}(\bar \Omega)$ for every $\ell>0$, where $u_\infty$ solves \eqref{eq0}.
\end{proposition}
\begin{proof}
This follows from Lemma~\ref{la:LinftyRHSest}, from the estimates in \cite{gru} (compare also to \cite{RS}), and the theorem of Arzel\`a-Ascoli.%\footnote{\armin{Do we have $\sup_k \lambda_k < \infty$?} \luca{Good that you noticed that. We don't have this in general, we have to assume it. I changed the statement of the theorem.}}
\end{proof}

We shall now assume that, up to a subsequence, $m_k\to \infty$ as $k\to\infty$  and we consider  $x_{k} \in {\Omega}$ so that 
\begin{equation}\label{eq:xkdef}
m_k \equiv \sup_{ \Omega}u_{k}=u_{k}(x_{k}) \to\infty \quad \text{as }k\to\infty.
\end{equation}
%Take $v_{k}=m_{k}^{-1}u(r_{k}x+x_{k})$ where $r_{k}$ is chosen as in \eqref{eq:rkdef}.
%We will assume from now on that $m_{k}\to +\infty$ and $\|u_{k}\|_{X}\leq C$.

\subsection{Rescaling and Convergence}

\begin{lemma}\label{rk}
Let $r_k$ and $x_k$ be defined by \eqref{eq:rkdef} and \eqref{eq:xkdef} respectively. Then we have
$$\lim_{k\to \infty} \frac{\dist(x_{k},\partial \Omega)}{r_{k}}=+\infty$$
\end{lemma}
\begin{proof}
For the sake of contradiction, we assume that $$\lim_{k\to \infty} \frac{\dist(x_{k},\partial \Omega)}{r_{k}}<\infty.$$
Let us assume that 
\begin{equation}\label{limrk}
0 < \lim_{k\to \infty} \frac{\dist(x_{k},\partial \Omega)}{r_{k}} < \infty.
\end{equation}
If the above limit vanishes then the argument is similar.
We set $\Omega_{k}=\{r_{k}^{-1}(x-x_{k});x\in \Omega\}$. Then
$$v_{k}(x):=\frac{u_k(r_k x+x_k)}{m_k}$$
satisfies
\begin{equation}
\begin{cases}
(-\Delta)^{\frac{n}{2}}v_{k}= \frac{2^n(n-1)!}{m_{k}^{2}}v_{k}e^{\frac{n}{2}m_{k}^{2}(v_{k}^{2}-1)} \quad &\text{ in $\Omega_{k}$}\\
v_{k}\in X(\Omega_k). \quad &
\end{cases}
\end{equation}
Notice that,
\[
 \vrac{(-\Delta)^{\frac{n}{4}}v_{k}}_{L^2(\R^n)} = (m_k)^{-1} \vrac{(-\Delta)^{\frac{n}{4}}u_{k}}_{L^2(\R^n)} \xrightarrow{k \to \infty} 0.
\]
Then by the Sobolev embedding, Proposition~\ref{pr:Sobolev} using also \eqref{Rieszprop}, the boundedness of the Riesz transform, and that $\laps{1} = \lapms{\frac{n}{2}-1} (-\lap)^{\frac{n}{4}}$,
\begin{equation}\label{gradvk}
\begin{split}
 &\vrac{\nabla v_{k}}_{L^n(\R^n)}= c\vrac{\Rz \laps{1} v_{k}}_{L^n(\R^n)} \aleq \vrac{ \laps{1} v_{k}}_{L^n(\R^n)}\\
 \aleq & \vrac{\lapms{\frac{n}{2}-1} (-\lap)^{\frac{n}{4}} v_{k}}_{L^n(\R^n)} \leq  C\vrac{(-\Delta)^{\frac{n}{4}}v_{k}}_{L^2(\R^n)} \xrightarrow{k \to \infty} 0.
\end{split}
 \end{equation}
On the other hand, by \eqref{limrk} there exists some $R > 0$ so that $B_{4R}(0) \subset \Omega_k$ for all $k \in \N$. Then
\[
\vrac{(-\Delta)^{\frac{n}{2}}v_{k}}_{L^\infty(B_{3R}(0))} \xrightarrow{k \to \infty} 0.
\]
This implies that for any $\alpha \in (0,n)$, Lemma~\ref{la:LinftyRHSest},
\[
 [v_{k}]_{C^{\alpha}(B_{2R}(0))} \le  C.
\]
So recalling that $|v_k|\le 1$, by Arzel\`a-Ascoli (up to a subsequence) we have that $v_k\to v$ in $C^{n-1}(B_R)$ for some $v$. 
Since at the same time $\nabla v_k \to 0$ in $L^n(\R^n)$ and $v_k (0) = 1$, we know that $v \equiv 1$ in $B_R$.

On the other hand, take $R_1 > R$ so that $B_{\frac{R_1}{2}}(0) \cap \partial \Omega_k \neq 0$ for all but possibly finitely many $k \in \N$. Using \eqref{gradvk}, and noticing that $v_k \equiv 0$ on a fixed part of positive measure of $B_{R_1}$, we know that $v_k \to 0$ in $L^n(B_{R_1}(0))$, hence $v \equiv 0$ in $B_R$. This contradicts $v \equiv 1$.
\end{proof}

%\begin{lemma}\label{diff} We have
%\begin{equation}
%u_{k}(x_{k}+r_{k}x)-m_{k}\to 0 \text{ in } C^{2}_{loc}(\mathbb{R}^{3}) \text{ as }k\to \infty.
%\end{equation}
%\end{lemma}

\begin{lemma}\label{diff} Let $m_k$ be as in \eqref{eq:xkdef}. Then we have
\begin{equation}
u_{k}(x_{k}+r_{k}x)-m_{k}\to 0 \text{ in } C^{n-1}_{\loc}(\mathbb{R}^{n}) \text{ as }k\to \infty.
\end{equation}
\end{lemma}

\begin{proof}
Let $\tilde{u}_k:=u_{k}(x_{k}+r_{k}x)$. Then $\tilde{u}_k \in C^0_c(\R^n)\cap X(\Omega_k)$ and
\[
 \sup_{x \in \R^n} |\tilde{u}_k(x)| = \tilde{u}_k(0) = m_k \in [0,\infty).
\]
As above by Sobolev embedding, $\tilde{u}_k \in W^{1,n}_0(\Omega_k)$
\[\begin{split}
 \limsup_{k\to\infty}\|\nabla  \tilde{u}_k\|_{L^n(\R^n)} &\le C \limsup_{k\to\infty} \|(-\lap)^\frac n4 \tilde{u}_k\|_{L^2(\R^n)} \\
&= C\limsup_{k\to\infty}\|(-\lap)^\frac n4 u_k\|_{L^2(\R^n)}\\
& \leq C(\Lambda).
\end{split}
\]
and from \eqref{Rieszprop-1}, \eqref{Rieszprop0} and \eqref{Rieszprop} below for $k$ large enough we get
\begin{equation}\label{gkbound}
\begin{split}
\|\laps{1} \tilde{u}_k\|_{L^n(\R^n)} &=  \sum_{i=1}^n\|\Rz_i\Rz_i \laps{1} \tilde{u}_k\|_{L^n(\R^n)}\\
&\le C \sum_{i=1}^n\|\Rz_i \laps{1} \tilde{u}_k\|_{L^n(\R^n)}  \\
&\le C \|\nabla  \tilde{u}_k\|_{L^n(\R^n)} \\
& \leq C(\Lambda).
\end{split}
\end{equation}
Notice that 
$$|(-\Delta)^\frac n2 \tilde{u}_k|\le \frac{C}{m_k}\quad \text{in }\Omega_k.$$
Finally, by Lemma~\ref{rk} for any $\varphi \in C_c^\infty(\R^3)$, for all sufficiently large $k$ depending on the size of the support of $\varphi$,
\begin{equation}\label{eq:pdeforwk}
 \bigg|\int_{\R^n} (-\lap)^{\frac{n}{4}} \tilde{u}_k\, (-\lap)^{\frac{n}{4}} \varphi\, dx\bigg| \leq C\ \frac{1}{m_k} \int_{\R^3} |\varphi|\,dx.
\end{equation}
Let $g_k := \laps{1} \tilde{u}_k$, bounded in $L^n(\R^n)$, according to \eqref{gkbound}.
There is a weakly convergent subsequence $g_k \rightharpoonup g$ in $L^n(\R^n)$. Moreover, we have for any $\varphi \in C_c^\infty(\R^n)$, by \eqref{eq:pdeforwk}
\[
  \int_{\R^n} g\, (-\lap)^\frac{n-1}{2} \varphi \, dx= \lim_{k \to \infty} \int_{\R^n} g_k\, (-\lap)^\frac{n-1}{2} \varphi \, dx=  \lim_{k \to \infty} \int_{\R^n} (-\lap)^\frac n4 \tilde{u}_k\, (-\lap)^\frac n4 \varphi\, dx \xrightarrow{k \to \infty} 0.
\]
Consequently, $g \in C^\infty(\R^n)\cap L^n(\R^n)$, and pointwise $(-\lap)^\frac{n-1}{2}g \equiv 0$. This implies that $g\equiv 0$. Indeed by elliptic estimates (see e.g. \cite[Proposition 4]{mar1}) and H\"older's inequality it follows that 
$$\|g\|_{L^\infty(B_1)}\le C\|g\|_{L^1(B_2)}\le \tilde C \|g\|_{L^n(B_2)},$$
which scaled gives
$$\|g\|_{L^\infty(B_R)} \le\tilde C R^{-1}\|g\|_{L^n(B_{2R})}\to 0\quad \text{as }R\to\infty.$$

So we have obtained, that $\laps{1} \tilde{u}_k \rightharpoonup 0$ in $L^n(\R^n)$. Then, using \eqref{Rieszprop} and \eqref{eq:RiezPI} we also have
$$\nabla \tilde{u}_k = \Rz \laps{1} \tilde{u}_k \rightharpoonup 0\quad \text{in }L^n(\R^n).$$
Since $\tilde{u}_k$ is uniformly bounded in $H^\frac n2(\R^n)$, since $n\geq 3$, we also have strong convergence in $W^{1,2}_{\loc}(\R^n)$. In particular up to choosing a subsequence, for any $R > 1$, 
\begin{equation}\label{grad0}
\nabla \tilde{u}_k \to 0 \quad \text{in }L^2(B_R).
\end{equation}
On the other hand, observe the following: For any $R > 1$, for all large $k \in \N$,  we have  $B_{2R}\subset \Omega_k$. From \eqref{eq:pdeforwk},  Lemma~\ref{la:LinftyRHSest} we obtain
\[
 \|\nabla \tilde{u}_k\|_{C^{n-2,\alpha}(B_R)} \leq C
\]
for a uniform constant $C$ and $\alpha\in (0,1)$.

Since $\tilde{u}_k(0) = m_k$, we have
\[
\|\tilde{u}_k-m_k\|_{L^\infty(B_R)} \leq  \|\nabla \tilde{u}_k\|_{L^\infty(B_R)} \leq C,
\]
and consequently we have shown that
\[
\|\tilde{u}_k-m_k\|_{C^{n-1,\alpha}(B_R)} \leq C
\]
Now Arzel\`a-Ascoli gives (up to a further subsequence) $C^{n-1}(B_R)$-convergence of $\tilde{u}_k-m_k$, and using \eqref{grad0} we have that $\tilde{u}_k - m_k \to 0$  in $C^{n-1}(B_R)$.  Since $R$ is arbitrary the proof is complete.
\end{proof}

\subsection{Gradient-type estimates}
Note that from \eqref{energy}
\[
 \limsup_{k \to \infty} \|u_k \laps{n} u_k\|_{L^1(\Omega)} \leq \Lambda.
\]
Moreover, as in \cite[Proof of Lemma 5]{mar4}, we know that for the Orlicz space $L \log^{\frac{1}{2}} L(\Omega)$,
\[
 \limsup_{k \to \infty} \|\laps{n} u_k \|_{L\log^{\frac{1}{2}} L(\Omega)} \leq C(\Lambda,\Omega).
\]
We will now need the following crucial estimate applied to $u=u_k$ and $\rho=Rr_k$ for a given $R>0$ and $k$ so large that $B_{10\rho} (x_k) \subset \Omega$ (compare to Lemma \ref{rk}).

\begin{proposition}\label{la:ulapuest}
Let $\Omega$ be a smoothly bounded domain, and consider $u \in X(\Omega)$ such that $\laps{n} u = f$ weakly in $\Omega$ for some $f \in L\log^{\frac{1}{2}} L(\Omega)\cap L^\infty(\Omega)$. Assume moreover that
\begin{equation}\label{eq:uniformbound}
 \vrac{u\, f}_{L^1(\Omega)} +  \|f\|_{L\log^{\frac{1}{2}} L(\Omega)} + \vrac{(-\lap)^{\frac{n}{4}} u}_{L^2(\R^n)} \leq C_1.
\end{equation}
Then for a constant depending $C_2$ depending only on $C_1$ and $s \in (0,n)$ we have %for any $B_{10\rho}(x) \subset \Omega$,
\[
 \sup_{B_{4\rho}(x_0) \subset \Omega} \rho^{s-n} \int_{B_{\rho}(x_0)}|u\laps{s} u|\, dx \leq C_2
\]
\end{proposition}
\begin{proof}
We will use the Lorentz spaces $L^{(p,q)}$, for which we refer the reader to the appendix. Using the H\"older-type inequality (see \cite{one})
$$\|gh\|_{L^1(\Omega)}\le \|g\|_{L^{(\frac{n}{n-s},1)}(\Omega)}\, \|h\|_{L^{(\frac{n}{s},\infty)}(\Omega)},$$
we get (for $B_\rho=B_\rho(x_0)$, to simplify the notation)
\[\begin{split}
\rho^{s-n} \int_{B_{\rho}}|u\laps{s} u| dx &\le \rho^{s-n} \|\chi_{B_\rho}\|_{L^{(\frac{n}{n-s},1)}(B_\rho)}  \vrac{u\laps{s} u}_{L^{(\frac{n}{s},\infty)}(B_\rho)}\\
&= C  \vrac{u\laps{s} u}_{L^{(\frac{n}{s},\infty)}(B_\rho)},
\end{split}\]
so that it remains to show the bound
\[
\sup_{B_{4\rho} \subset \Omega} \vrac{u\laps{s} u}_{L^{(\frac{n}{s},\infty)}(B_\rho)} \leq C_2.
\]
For $\eps > 0$ we denote with $u^\eps \in C_c^\infty(\R^n)$ the usual mollification.

Consider now a cut-off function $\theta_{B_1}\in C^\infty(B_2)$, $\theta_{B_1}\equiv 1$ in $B_1$ and $0\le \theta_{B_1}\le 1$ everywhere. Set $\theta_{B_{2\rho}}:=\theta_{B_1}(\cdot/2\rho)\in C^\infty_c(B_{4\rho})$.
Then since $u^\eps \in C^\infty_c(\R^n)$ we have for $s \in (0,n)$ pointwise in $B_\rho$:
%\begin{align*}
% |u\laps{s} u^\eps| = \lapms{t} (\laps{t} u)\, \laps{s} u^\eps
%\end{align*}
%Duality $g \in L^{\frac{n}{s}}$,
%\[
%\int g\, \lapms{t} (\laps{t} u)\, \laps{s} u^\eps\, dx = \int (\laps{t} u)\, \lapms{t} (g \laps{s} u^\eps) dx
%\]
\begin{align*}
 |u\laps{s} u^\eps| = &| u\lapms{n-s}\laps{n} u^\eps| \\
 \le & |u\lapms{n-s}(\theta_{B_{2\rho}}\laps{n} u^\eps)|+|u\lapms{n-s}((1-\theta_{B_{2\rho}})\laps{n} u^\eps)|\\ 
 \le&|u\lapms{n-s}(\theta_{B_{2\rho}}\laps{n} u^\eps)|\\
&+|u\lapms{n-s}((1-\theta_{B_{2\rho}})\laps{n} u^\eps)-\lapms{n-s}(u (1-\theta_{B_{2\rho}})\laps{n} u^\eps)|\\
&+|\lapms{n-s}(u \laps{n} u^\eps)|\\
=:& I + II + III.
\end{align*}
By Proposition~\ref{borderline1} and Proposition~\ref{borderline2}, using that $u = \lapms{\frac{n}{2}}(-\lap)^{\frac{n}{4}} u$ we infer
\begin{equation}\label{eq:estI}
 \vrac{I}_{L^{(\frac {n}{s},\infty)}(\R^n)} \aleq \vrac{(-\lap)^\frac n4 u}_{L^2(\R^n)} \vrac{\laps{n} u^\eps}_{L\log^{\frac{1}{2}} L(B_{4\rho)}} .
\end{equation}
From the disjoint-support commutator estimate, see Proposition~\ref{la:disjcommutator}, we have
\begin{equation}\label{eq:estII}
\vrac{II}_{L^{(\frac {n}{s},\infty)}(B_\rho)} \aleq \| (-\lap)^{\frac{n}{4}} u \|_{L^2(\R^n)}^2.
\end{equation}
Since the support of $u$ is contained in $\Omega$, by the Sobolev inequality
\begin{equation}\label{eq:estIII}
 \vrac{III}_{L^{(\frac{n}{s},\infty)}(\R^n)} \aleq \vrac{u \laps{n} u^\eps}_{L^1(\Omega)}
\end{equation}

Combining the estimates \eqref{eq:estIII}, \eqref{eq:estI}, \eqref{eq:estII} we arrive at
\[
\begin{split}
 \|u\laps{s} u^\eps\|_{L^{(\frac{n}{2},\infty)}(B_\rho)} &\aleq \|u \laps{n} u^\eps \|_{L^1(\Omega)} + \|(-\lap)^{\frac{n}{4}} u\|_{L^2(\R^n)}\ \|\laps{n} u^\eps\|_{L\log^{\frac{1}{2}} L(B_{4\rho})} \\
 &\quad + \| (-\lap)^{\frac{n}{4}} u\|_{L^2(\R^n)}^2\\
 &\overset{\eqref{eq:uniformbound}}{\leq} 
\|u \laps{n} u^\eps \|_{L^1(\Omega)} + C_1\ \|\laps{n} u^\eps\|_{L\log^{\frac{1}{2}} L(B_{4\rho})} + (C_1)^2.
 \end{split}
 \]
It remains to take $\eps \to 0$, but some care is needed, since $\laps{n} u$ is in general not a function, but a distribution.
 
Firstly, since $B_{4\rho} \subset \Omega$, for $\eps < \rho$ we have that
\[
 \laps{n} u^\eps = (\laps{n} u)^\eps \text{\quad in } B_{4\rho}.
\]
In particular, for $\eps < \rho$
\[
 \|\laps{n} u^\eps\|_{L\log^{\frac{1}{2}} L(B_{4\rho})} \aleq \|\laps{n} u\|_{L\log^{\frac{1}{2}} L(B_{4\rho})} \leq \|\laps{n} u\|_{L\log^{\frac{1}{2}} L(\Omega)} \leq C_1.
\]
For the remaining term $\|u \laps{n} u^\eps \|_{L^1(\Omega)}$, we need to argue as follows. Firstly, since $u^\eps$ is the usual mollification, we have
\begin{equation}\label{eqeps}
 \|\laps{n} u^\eps\|_{L^2(\R^n)} \leq \eps^{-\frac{n}{2}} \| (-\lap)^{\frac{n}{4}} u\|_{L^2(\R^n)}
\end{equation}
Moreover, since $\Omega$ is smooth and bounded and $u \in X(\Omega)$, the results by \cite{gru}, see also \cite[Theorem~1.2]{XR2}, using that 
\[
\|\laps{n} u\|_{L^\infty(\Omega)} =: C_3 < \infty,
\]
then if we set $\Omega_{-\eps} := \{x \in \Omega: \dist(x,\partial \Omega) > \eps\}$
\[
 \|u\|_{L^\infty(\Omega \backslash \Omega_{-\eps})} \aleq \eps^{\frac{n}{2}}\, C_3.
\]
In particular with \eqref{eqeps} we have 
\[\begin{split}
\|u\laps{n} u^\eps\|_{L^1(\Omega)} &\leq \|u\laps{n} u^\eps\|_{L^1(\Omega_{-\eps})} +  |\Omega \backslash \Omega_{-\eps}|^{\frac{1}{2}} \| (-\lap)^{\frac{n}{4}} u\|_{L^2(\R^n)}\\
&=  \|u\laps{n} u^\eps\|_{L^1(\Omega_{-\eps})} +  o(1) \text{\quad as $\eps \to 0$.}
\end{split}
\]
Now note again that
\[
 \laps{n} u^\eps = (\laps{n} u)^\eps \quad \mbox{pointwise in $\Omega_{-\eps}$}.
\]
Consequently, 
\[
 \|u\laps{n} u^\eps\|_{L^1(\Omega)} \leq \|u(\laps{n} u)^\eps\|_{L^1(\Omega_{-\eps})} +o(1)
\xrightarrow{\eps \to 0} \|u\ \laps{n} u\|_{L^1(\Omega)}.
\]
This concludes the proof of Lemma~\ref{la:ulapuest}.
\end{proof}

\subsection{Convergence of \texorpdfstring{$\eta_k$}{etak}}
Let $\eta_k$ be as in \eqref{defetak}.
%\[
% \eta_{k}(x):=m_{k}(u_{k}(x_{k}+r_{k}x)-m_{k}).
%\]
From Proposition~\ref{la:ulapuest} we now infer:

\begin{proposition}\label{gradest}
For every $s\in (0,n)$ there exists $C>0$ such that for every $R>0$ and $k$ large enough (depending on $R$ and $s$) we have
\begin{equation}\label{eq:lapsest}\int_{B_{R}}|(-\Delta)^\frac{s}{2} \eta_{k}| dx \leq C R^{n-s}.\end{equation}
\end{proposition}
\begin{proof}
According to Lemma \ref{diff} we have
$$m_k\le 2u_k\quad \text{on } B_{Rr_k}(x_k)\text{ for }k\text{ large enough,}$$
hence with Proposition~\ref{la:ulapuest} applied with $u=u_k$ and $\rho=Rr_k$ we obtain (note that $\laps{s} (m_k^2) = 0$)
\[
\begin{split}
\int_{B_{R}}|(-\Delta)^\frac{s}{2} \eta_{k}| dx&= \frac{m_k}{r_{k}^{n-s}} \int_{B_{Rr_{k}}(x_k)}|(-\Delta)^\frac s2 u_{k}| dx \\
&\le \frac{2}{r_{k}^{n-s}} \int_{B_{Rr_{k}}(x_k)}|u_k (-\Delta)^\frac s2 u_{k}| dx\\
& \leq C R^{n-s},
\end{split}
\]
as claimed.
\end{proof}

% % 
% % \begin{proof}
% % 
% % We fix $R>0$. Then using Propositions \ref{est1} and \ref{est2}, we have,
% % 
% % $$\int_{B_{Rr_{k}}(x_{k})}|\nabla^{l}(u_{k}^{2})(x)|dx\leq C (Rr_{k})^{3-l}.$$
% % On the other hand, from (\ref{diff}), we have that for $x\in B_{Rr_{k}}(x_{k})$,
% % $$m_{k}|\nabla u_{k}|(x)\leq C u_{k}(x) |\nabla u_{k}|$$
% % Hence  $$u_{k}(x_{k})\int_{B_{Rr_{k}}}|\nabla u_{k}(x)| dx \leq C (Rr_{k})^{2}.$$
% % A similar inequality hold for the second derivative.
% % \end{proof}
% 
% Let
% \[
%  \eta_{k}(x):=m_{k}(u_{k}(x_{k}+r_{k}x)-m_{k})
% \]
% 

\begin{proposition}\label{pr:boundedness}
For every $B_R \subset \R^n$ and any $\alpha \in [0,1)$ there exists a constant $C_{R,\alpha}$ so that
\[
 \| \eta_k \|_{C^{n-1+\alpha}(B_R)} \leq C_{R,\alpha}.
\]
for $k$ large enough.
\end{proposition}

\begin{proof} We have that $|(-\Delta)^\frac{n}{2} \eta_k|\le C(R)$ in $B_R$, in the sense that
\[\bigg|\int_{\R^n}(-\Delta)^\frac{n}{4} \eta_k (-\Delta)^\frac{n}{4} \varphi dx\bigg| \le C\|\varphi\|_{L^1(B_R)},\quad \text{for }\varphi\in C^\infty_c(B_R).\]
This can be rewritten as
\begin{equation}\label{eqinfty}
\bigg|\int_{\R^n}(-\Delta)^\frac{1}{2} \eta_k (-\Delta)^\frac{n-1}{2} \varphi dx\bigg| \le C\|\varphi\|_{L^1(B_R)},\quad \text{for }\varphi\in C^\infty_c(B_R),
\end{equation}
which means that the function $\psi_k:=(-\Delta)^\frac12 \eta_k$ satisfies
$$|(-\Delta)^\frac{n-1}{2}\psi_k| \le C_R\quad \text{in }B_R$$
in the sense of distributions (notice that $(-\Delta)^\frac{n-1}{2}$ is an integer power of $-\Delta$ since $n$ is odd).
This, together with the estimate
$$\|\psi_k\|_{L^1(B_R)}\le C R^{n-1}$$
given by Proposition \ref{gradest}, and standard elliptic estimates (see e.g. Proposition 4 and Lemma 20 in \cite{mar1}) implies that
$$\|\psi_k\|_{C^{n-2,\alpha}(B_{R/2})}\le C_{R,\alpha}\quad \text{for }0\le \alpha<1,$$
as claimed. Together with Harnack's inequality (see \cite{IMM}) we get
$$\|\eta_k\|_{C^{n-1,\alpha}(B_{R/4})}\le C_{R,\alpha}\quad \text{for }0\le \alpha<1,$$
and replacing $R$ with $4R$ we conclude.
% Split
%$\eta_k= \eta_{1,k}+\eta_{2,k}$ with $\eta_{1,k}\in H^\frac{1}{2}_{00}(B_{R/2})$ solving
%$$(-\Delta)^\frac12 \eta_{1,k}=(-\Delta)^\frac12 \eta_{k}\quad \text{in }B_{R/2}$$
%and $\eta_{2,k}\in C^\infty(B_R)$ satisfying
%$$(-\Delta)^\frac12 \eta_{2,k}=0\quad \text{in }B_{R/2},\quad \eta_{2,k}=\eta_k\quad \text{in }\R^n\setminus B_{R/2}.$$
%By (nonlocal) elliptic estimates one gets
\end{proof}

\begin{proposition}\label{boundLe}
The sequence $(\eta_k)$ is uniformly bounded in $L_s(\R^n)$ for any $s > 0$.
\end{proposition}
\begin{proof}
Since by Proposition \ref{pr:boundedness} the sequence $(\eta_k)$ is bounded in $L^\infty(B_1)$, it is easy to see that boundedness of $(\eta_k)$ in $L_s(\R^n)$ for some $s>0$ implies boundedness in $L_{s'}(\R^n)$ for every $s'>s$. Therefore without loss of generality we can assume that $s<1$. We then have
\begin{equation}\label{lapshwk}
\laps{s} \eta_k(x) = C_{n,s} \int_{\R^n} \frac{\eta_k(y)+\eta_k(x)}{|x-y|^{n+s}}\, dy.
\end{equation}
Consequently,  for an arbitrary $\varphi\in C^\infty_c(B_1)$ using \eqref{eq:lapsest}
\[
\begin{split}
C\|\varphi\|_{L^\infty(B_1)}&\overset{\eqref{eq:lapsest}}{\ge}%\bigg|\int_{\R^n} (-\Delta)^\frac n4 \eta_k\ (-\Delta)^\frac n4 \varphi\, dx\bigg|\\&=
\bigg|\int_{B_1} \laps{s} \eta_k\, \varphi\, dx\bigg|\\
&\overset{\eqref{lapshwk}}{\geq}
\int_{B_1} \int_{B_{2}}\frac{\eta_k(x)-\eta_k(y)}{|x-y|^{n+s}}dy\, \varphi(x)\, dx\\
&\quad +\int_{B_1} \int_{B_{2}^c}\frac{1}{|x-y|^{n+s}}dy\, \eta_k(x)\,  \varphi(x)\, dx \\
&\quad +\int_{B_1} \int_{B_{2}^c}\frac{-\eta_k(y)}{|x-y|^{n+s}}dy\, \varphi(x)\, dx\\
% 
% -c\int_{B_1} \int_{|h| < 2}\frac{|\eta_k(x+h)+\eta_k(x-h)-2\eta_k(x)|}{|h|^{n+1}}\, dh\ (-\Delta)^\frac{n-1}2 \varphi(x)\, dx\\
% &\quad -c\int_{B_1} \int_{|h| > 2}\frac{1}{|h|^{n+1}}\, dh\ 2|\eta_k(x)| |(-\Delta)^\frac{n-1}2 \varphi(x)|\, dx \\
% &\quad +c|\int_{B_1} \int_{|h| > 2} \frac{\eta_k(x+h)+\eta_k(x-h)}{|h|^{n+1}}\, dh\ (-\Delta)^\frac{n-1}2 \varphi(x)\, dx|\\
&=:I+II +III.
\end{split}
\]
Since by Proposition~\ref{pr:boundedness},
\[
 |\eta_k(x) - \eta_k(y)| \aleq |x-y| \quad \forall x,y \in B_2,
\]
we have that
\[
 |I| \aleq \int_{B_1} |\varphi(x)| \int_{|x-y| \leq 3} |x-y|^{-n+1-s}\ dy\ dx \aleq \|\varphi\|_{1}. 
\]
Since we also have $|\eta_k(x)| \leq C$ for all $x \in B_2$,
\[
 |II| \aleq \int_{B_1} |\varphi(x)| \int_{|x-y| > 1} |x-y|^{-n-s}\ dy\ dx \aleq \|\varphi\|_{1}. 
\]
Finally, since $-\eta_k(y) = |\eta_k(y)|$, we arrive at
\[
 \int_{B_1} \int_{B_{2}^c}\frac{|\eta_k(y)|}{|x-y|^{n+s}}dy\ \varphi(x)\, dx \leq C,
\]
for a constant depending on $\varphi$ and $s$, but independent of $k$. Taking $\varphi(x)$ to be non-negative and so that $\varphi \equiv 1$ on $B_{1/2}$, we arrive at
\[
 \int_{|y|> 2}\frac{|\eta_k(y)|}{1+|y|^{n+s}}dy \leq C.
\]
Since again by Proposition~\ref{pr:boundedness} for a $C$ uniform in $k$,
\[
 \int_{|y|< 2}\frac{|\eta_k(y)|}{1+|y|^{n+s}}dy \leq C.
\]
we have shown that
\[
 \sup_{k \in \N} \|\eta_k\|_{L_s(\R^n)} \leq C.
\]
\end{proof}

\begin{proposition}
Up to a subsequence,  $\eta_{k} +\log 2\to \eta_0 =\log(\frac{2}{1+|\cdot |^{2}})$ in $C^{\ell}_{\loc}(\mathbb{R}^n)$ for every $\ell\ge 0$, and
\begin{equation}\label{eneretak}
\lim_{R\to\infty}\lim_{k\to\infty} \int_{B_{Rr_k}(x_k)}\lambda_k u_k^2 e^{\frac n2 u_k^2}dx =(n-1)!\int_{\R^n} e^{n\eta_0}dx =\Lambda_1.
\end{equation}
\end{proposition}

\begin{proof}
Let $\eta_0$ be the pointwise limit of $\eta_k +\log 2$, which exists up to a subsequence, by Proposition \ref{pr:boundedness} and Arzel\`a-Ascoli's theorem. In fact the limit is in $C^\ell_{\loc}(\R^n)$ for every $\ell\ge 0$ since with Proposition \ref{boundLe} one can bootstrap regularity for the operator $(-\Delta)^\frac n2$, see e.g. \cite[Corollary 24]{JMMX}. It follows from Proposition \ref{boundLe} that 
\begin{equation}\label{eta0s}
\eta_0 \in L_s(\R^n)\quad \text{for every }s>0.
\end{equation}
We then have 
$$\lim_{k\to\infty}\int_{\R^n}e^{n(\eta_k+\log 2)} \varphi dx =\int_{\R^n}e^{n\eta_0}\varphi dx \quad  \text{for every }\varphi \in {C^\infty_c(\R^n)}.$$
We will show that moreover
\begin{equation}\label{eq:convergencePDEgoal}
  \lim_{k \to \infty} \int_{\R^n} \eta_k\ \laps{n} \varphi dx =  \int_{\R^n} \eta_0 \laps{n} \varphi dx\quad \text{for every }\varphi \in {C^\infty_c(\R^n)}.
\end{equation}
Then $\eta_0$ satisfies $(-\Delta)^\frac n2 \eta_0=(n-1)!e^{n\eta_0}$ as a distribution, and in fact also as tempered distribution. Then from to Theorem \ref{thmali} we infer that $\eta_0=v+P$ where $|v|\le C(1+\log(1+|\cdot|))\in L_{s}(\R^n)$ for every $s>0$, and $P$ is polynomial bounded from above. It is easy to see that if $P$ is not constant, then $P\not\in L_s(\R^n)$ for any $s>0$, which contadicts \eqref{eta0s}. Therefore $P$ is constant and $\eta_0$ is as in \eqref{spherical}. It remains to determine $\lambda$ and $x_0$ in \eqref{spherical}, but this is easy since $\eta_k(0)=0=\max_{\R^n}\eta_k$, so that $\eta_0(0)=\log 2=\max_{\R^n}\eta_0$, i.e. $x_0=0$, $\lambda=1$ and
%Also, considering that in Proposition \ref{gradest} we have (\ToDo for what needed ?, doesnt $L_\eps$ already rule out freaky polynomials?)
%YES
%$$\int_{B_R}|\Delta\eta_0|dx\le C R^{n-2},$$
$\eta_0(x)=\log\frac{2}{1+|x|^2}$.

In order to obtain \eqref{eq:convergencePDEgoal}, Assume that for some $R > 0$, $\operatorname{supp} \varphi \subset B_R$ and let $\theta > 1$. Then,
\[
\begin{split}
\int \eta_0 \laps{n} \varphi\, dx  - \int \eta_k \laps{n} \varphi\, dx
=&\int_{B_{\theta R}} (\eta_0 \laps{n} \varphi  - \eta_k\ \laps{n} \varphi)\, dx\\
&+\int_{\R^n \backslash B_{\theta R}} (\eta_0 \laps{n} \varphi  - \eta_k\ \laps{n} \varphi)\, dx\\
=:& I + II.
\end{split}
 \]
Notice that by the disjoint support $\varphi$, see Lemma~\ref{la:nonloc},
\[
\begin{split}
|II| &\aleq \|\varphi\|_{L^1(\R^n)}\, \int_{\R^n \backslash B_{\theta R}} \frac{|\eta_0(x)|+|\eta_k(x)|}{|x|^{2n}}\, dx\\
&
\aleq \|\varphi\|_{L^1(\R^n)}\, \theta^{s-n} \int_{\R^n \backslash B_{\theta R}} \frac{|\eta_0(x)|+|\eta_k(x)|}{1+|x|^{n+s}}\, dx
\end{split}
\]
and the uniform bound of $\eta_0$ and $\eta_k$ in $L_{s}(\R^n)$ implies that
\[
 |II| \aleq \theta^{s -n},
\]
for a constant independent of $k$. On the other hand, $\eta_0 - \eta_k \to 0$ uniformly in $B_{\theta R}$, which implies that $\lim_{k \to \infty} I = 0$. Consequently,
\[
 \lim_{k\to \infty}\bigg|\int \eta_0 \laps{n} \varphi\, dx - \int \eta_k\ \laps{n} \varphi\, dx\bigg| \aleq  \theta^{s -n},
\]
for any $\theta > 1$, and letting $\theta \to \infty$ we conclude the proof of \eqref{eq:convergencePDEgoal}.

Finally, using Lemma \ref{diff} and the definition of $r_k$, we obtain
\[\begin{split}
\int_{B_{Rr_k}(x_k)}\lambda_k u_k^2 e^{\frac n2 u_k^2}dx&=\int_{B_R}r_k^n \lambda_k u_k^2(r_k\cdot)e^{\frac n2 m_k^2}e^{\frac n2(u_k(r_k\cdot)-m_k)^2}e^{n\eta_k}dx\\
&=\int_{B_R}r_k^n \lambda_k m_k^2(1+o(1))e^{\frac n2 m_k^2}e^{o(1)}e^{n\eta_k}dx\\
&=2^n(n-1)!\int_{B_R}(1+o(1))e^{n\eta_k}dx\\
&=(n-1)!\int_{B_R}e^{n\eta_0}dx+o(1)
\end{split}\]
with $o(1)\to 0$ as $k\to\infty$. Now letting also $R\to\infty$ and noticing that
$$\int_{\R^n}e^{n\eta_0}dx=|S^n|,$$
we infer \eqref{eneretak}.
\end{proof}

\section{Borderline and commutator estimates}\label{sec:borderline}
We have the following two borderline estimates:
\begin{proposition}\label{borderline1}
Let $g \in L^2(\R^n)$, $f \in L\log^{1/2}L(\R^n)$, $s \in (0,n)$. Then 
\[
\| \lapms{n-s} (f\, \lapms{\frac{n}{2}}g)\|_{(\frac{n}{s},\infty)} \aleq  \vrac{g}_{L^2} \vrac{f}_{L\log^{1/2}L} 
\]
\end{proposition}
\begin{proof}
By Fubini's theorem and Proposition~\ref{pr:orliczsobolev},
\[\begin{split}
\|f\lapms{\frac{n}{2}} g\|_{L^1}&\aleq  \int_{\R^n}\int_{\R^n}\frac{|g(z)|}{|z-y|^{n/2}}dz\ |f(y)|dy\\
&= \int_{\R^n}\int_{\R^n} \frac{|f(y)|}{|z-y|^{n/2}}dy\ |g(z)| dz\\
&=\||g|\ I_\frac{n}{2}|f|\|_{L^1}\\
&\le \|g\|_{L^2}\|I_\frac n2|f|\|_{L^2}\\
&\aleq\|g\|_{L^2}\|f\|_{L(\log^{1/2} L)},
\end{split}
\]
so that by Proposition~\ref{pr:Sobolev},
$$\|I_{n-s}(f\ I_{\frac n2}(g))\|_{(\frac ns,\infty)}\aleq \|f\ \lapms{\frac{n}{2}} g\|_{L^1}\aleq  \|g\|_{L^2}\|f\|_{L(\log^{1/2} L)}.$$
\end{proof}

\begin{proposition}\label{borderline2}
Let $g \in L^2(\R^n)$, $f \in L\log^{1/2}L(\R^n)$, $s \in (0,n)$. Then 
\[
\| \lapms{\frac{n}{2}} g\, \lapms{n-s}f\|_{(\frac{n}{s},\infty)} \aleq  \vrac{g}_{L^2} \vrac{f}_{L\log^{1/2}L} 
\]
\end{proposition}

\begin{proof} We write
$$\lapms{\frac{n}{2}} g\, \lapms{n-s}f(x)=\int_{\R^n}\int_{\R^n} k(x,y,z) g(z) f(y) dz dy,\quad k(x,y,z):=\frac{1}{|x-z|^\frac n2 |x-y|^s}.$$
For $\eps\in (0,s)$ we can now bound (cf. \cite{Sfracenergy})
$$k(x,y,z) \le |x-z|^{-\frac n2-\eps}|x-y|^{\eps-s}+ C|y-z|^{-\frac n2}|x-y|^{-s}=:II+III.$$
Indeed if $|y-z|\ge 2|x-z|$, then $|x-z|\le |x-y|$ (by the triangular inequality), hence $I\le II$. If $|y-z|\le 2|x-z|$, then $I\le C\,III.$
Therefore we have
$$|\lapms{\frac{n}{2}} g\, \lapms{n-s}f|\le I_{\frac n2-\eps} (|g|)\, I_{n-s+\eps}(|f|)+ CI_{n-s}(|f|\,I_{\frac n2}(|g|)).$$
The first term on the right-hand side can be bounded as
\[
\|  \lapms{\frac{n}{2}-\eps} g\, \lapms{n+\eps - s} f
\|_{(\frac{n}{s},\infty)} \aleq \|\lapms{\frac{n}{2}-\eps} (|g|)\|_{L^\frac{n}{\eps}}\, \|\lapms{n+\eps - s} (|f|)\|_{(\frac{n}{s-\eps},\infty)}
 \aleq \|g\|_{L^2}\,\|f\|_{L^1},
 \]
while the second term can be bounded by Proposition~\ref{borderline1}.
\end{proof}

\subsection{Disjoint-support estimates}

When $\operatorname{supp} \varphi \subset K$ for a compact set $K$ then in general we have no information on the support of $\laps{s} \varphi$, since $\laps{s}$ is a non-local operator. In particular $\laps{s} \varphi(x) \neq 0$ also for $x$ far away from $K$. However, there is a decay of $|\laps{s} \varphi(x)|$ as $\dist(x,K) \to \infty$. We shall call this pseudo-local behavior of $\laps{s}$.
It has been used allover the literature, for statements in the following form see \cite{BRS}.

\begin{definition}[Cut-off functions]\label{def:cutoff} With $\theta_{B_1}$ we will denote a fixed smooth function $\theta_{B_1}\in C^\infty_c(B_2)$ with $\theta_{B_1}\equiv 1$ in $B_1$ and $0\le \theta_{B_1}\le 1$ everywhere. Define
$$\theta_{B_\rho}(x):= \theta_{B_1}(x/\rho)\in C^\infty_c(B_{2\rho}),\quad \theta_{A_\rho}:=\theta_{B_\rho}-\theta_{B_\frac\rho2}\in C^\infty_c(B_{2\rho}\setminus B_{\frac \rho2}).$$
%$\theta_{A_1}:=\theta_{B_1}-\theta_{B_\frac12}$.
\end{definition}

In the proof of Proposition \ref{la:ulapuest} we used the following ``disjoint-support'' commutator estimate. 
Compared to the usual commutator estimates \cite{CRW,Seps} the estimates here are simpler, due to the disjoint support. Note that going through the proof, one may obtain a BMO-estimate, which is false for the commutator without disjoint support, see \cite{DLZ}.

\begin{proposition}\label{la:disjcommutator}
Define the commutator $[u,\lapms{t}](v) := u\lapms{t} v - \lapms{t} (uv)$.
Then for any $u, v \in C_c^\infty(\R^n)$,
\[
\vrac{ [u,\lapms{n-s}]((1-\theta_{B_{2\rho}})(-\Delta)^\frac n4 f)}_{L^{(\frac {n}{s},\infty)}(B_\rho)} \aleq \| (-\lap)^{\frac{n}{4}} u \|_{L^2(\R^n)}  \| f\|_{L^2(\R^n)}.
\]
\end{proposition}
\begin{proof}
By scaling we can assume that $\rho = 1$. With $\theta_\ell:=\theta_{A_{2^\ell}}\in C^\infty_c (B_{2^{\ell+1}} \backslash B_{2^{\ell-1}})$ as in Definition \ref{def:cutoff} we have pointwise in $\R^n$
\[
 1-\theta_{B_2} = \sum_{l=2}^\infty \theta_{\ell}.
\]
Moreover for $t \in (0,\infty)$ and $p \in [1,\infty]$ we have
\[
 \| \laps{t} \theta_{\ell} \|_{L^p(\R^n)} \leq C_{t,p}\ 2^{\ell (\frac{n}{p}-t)}.
\]
Since $u, v \in C_c^\infty(\R^n)$, we have then
\begin{equation}\label{eqcom}
 |[u,\lapms{n-s}]((1-\theta_{B_2})(-\Delta)^\frac n4 f)| \leq \sum_{\ell=2}^\infty |u\lapms{n-s}(\theta_{\ell}(-\Delta)^\frac n4 f)-\lapms{n-s}(u \theta_{\ell}(-\Delta)^\frac n4 f)|.
\end{equation}
Now set
\[
 u_{\ell} := \theta_{B_{2^{\ell+2}}} (u-(u)_{B_{2^{\ell+2}}}).
\]
Since $\theta_{B_{2^{\ell+2}}} \equiv 1$ in $B_{2^{\ell+2}}\supset B_1 \cup\operatorname{supp} \theta_{\ell}$, and the constant $(u)_{B_{2^{\ell+2}}}$ commutes with $\lapms{n-s}$, multiplying each term on the right-hand side of \eqref{eqcom} by $\theta_{B_{2^{\ell+2}}}$ and summing and subtracting the term
$$\theta_{B_{2^{\ell+1}}}(u)_{B_{2^{\ell+1}}} I_{n-s}(\theta_\ell (-\Delta)^\frac n4 f),$$ we find
\[\begin{split}
 |[u,\lapms{n-s}]((1-\theta_{B_2})(-\Delta)^\frac n4 f)|& \leq \sum_{\ell=2}^\infty |u_\ell\lapms{n-s}(\theta_{\ell}(-\Delta)^\frac n4 f)| + \sum_{\ell=2}^\infty |\lapms{n-s}(u_\ell \theta_{\ell}(-\Delta)^\frac n4 f)| \\
 &=: \sum_{\ell=2}^\infty ((I)_{\ell} + (II)_{\ell})\quad \text{in }B_1.
 \end{split}
\]
Now, by H\"older inequality and Lemma~\ref{la:freakshowestimate1}
\[
\begin{split}
 \vrac{u_\ell\lapms{n-s}(\theta_{\ell}(-\Delta)^\frac n4 f)}_{L^{(\frac{n}{s},\infty)}(B_1)} & \aleq \vrac{u_\ell}_{L^\frac{n}{s}(B_1)} \|\lapms{n-s}(\theta_{\ell}(-\Delta)^\frac n4 f)\|_{L^\infty(B_1)} \\
 &\aleq \vrac{u_\ell}_{L^\frac{n}{s}(B_1)}\ 2^{-\ell s} \vrac{f}_{L^2(\R^n)}\\
 \end{split}
\]
Note that for any $p < \infty$,
\[
 \vrac{u_\ell}_{p,\R^n} \aleq C_p 2^{\ell\frac{n}{p}} [u]_{BMO} \aleq 2^{\ell \frac{n}{p}} \vrac{(-\lap)^\frac n4 u}_{2,\R^n}.
\]
Taking $p>\frac{n}{s}$ and $\delta=\frac {n}{p}$
\[
 \vrac{u_\ell}_{\frac{n}{s},B_1}\aleq \vrac{u_\ell}_{p,B_1}\aleq 2^{\ell \delta}\ \vrac{(-\lap)^\frac n4 u}_{2,\R^n}.
\]
Together, we arrive at
\[
 \vrac{u_\ell\lapms{n-s}(\theta_{\ell}(-\Delta)^\frac n4 f)}_{(\frac{n}{s},\infty),B_1}
%  \aleq  2^{\ell \delta}\ 2^{\ell ((n-s)-n)}\ \vrac{(-\lap)^{n/4} u}_{2,\R^n}\ ^2 = 
\aleq 2^{\ell (\delta-s)}\ \vrac{(-\lap)^\frac n4 u}_{2,\R^n}\ \vrac{f}_{2,\R^n},
\]
and for $\delta < s$, this ensures
\begin{equation}\label{eq:sumII1l}
 \sum_{\ell =2}^\infty (I)_{\ell} \aleq \vrac{(-\lap)^\frac n4 u}_{2,\R^n}\ \vrac{f}_{2,\R^n}.
\end{equation}
It remains to treat $(II)_{\ell}$, and we do that with Lemma~\ref{la:freakshowestimate2}:
\begin{align*}
 \vrac{\lapms{n-s}(u_\ell\ \theta_{\ell}(-\Delta)^\frac n4 f)}_{(\frac{n}{s},\infty),B_1}
 \aleq &\vrac{\lapms{n-s}(u_\ell\ \theta_{\ell}(-\Delta)^\frac n4 f)}_{\frac{n}{s},B_1}\\
 \aleq &\max_{t \in [0,\frac{n}{2}]} \vrac{\laps{t}u_\ell}_2\ (2^\ell )^{t-\frac{n}{2} -s}\ \vrac{f}_{2,\R^n}%\\
% \aleq &\max_{t \in [0,\frac{n}{2}]} \vrac{\laps{t}u_\ell}_2\ (2^\ell )^{t-\frac{n}{2} -s}\ \vrac{(-\lap)^\frac n4 u}_{2,\R^n}. 
\end{align*}
Now for any $t \in [0,\frac{n}{2}]$, by the construction of $u_\ell$ and Poincar\'e and Sobolev-embeddings,
\[
 \vrac{\laps{t}u_\ell}_2 \aleq (2^\ell )^{\frac{n}{2}-t} \vrac{(-\lap)^\frac n4 u}_{L^2(\R^n)}.
\]
This leads to
\[
 \vrac{\lapms{n-s}(u_\ell \theta_{\ell}(-\Delta)^\frac n4 f)}_{(\frac{n}{2},\infty),B_1}
 %\aleq&\max_{t \in [0,\frac{n}{2}]} (2^\ell )^{\frac{n}{2}-t}\ (2^\ell )^{t-\frac{n}{2} -s}\ \vrac{(-\lap)^\frac n4 u}^2_{2,\R^n}\\
 \aleq 2^{-s\ell}\ \vrac{(-\lap)^\frac n4 u}_{2,\R^n}\ \vrac{f}_{2,\R^n}. 
\]
Again, this ensures
\begin{equation}\label{eq:estsumII2}
 \sum_{\ell =2}^\infty (II)_{\ell} \aleq \vrac{(-\lap)^\frac n4 u}_{2,\R^n}\ \vrac{f}_{2,\R^n}.
\end{equation}
\end{proof}

\begin{lemma}\label{la:nonloc}
Let $\varphi \in C_c^\infty(K)$ for some compact set $K$ and let $\Omega \subset \R^n$ be an open set containing $K$ with $\dist(\partial\Omega,K) \ge d$ for some $d > 0$. Then for any $p \in [1,\infty]$ and $s \in (0,\infty)$ we have
\[
 \vrac{\laps{s} \varphi}_{L^p(\R^n \backslash \Omega)} \le C_{n,s,p}d^{n-(n+s)p} \vrac{\varphi}_{L^1(K)},
\]
and for any $s \in (0,n)$ and $p > \frac{n}{n-s}$ we have
\[
 \vrac{\lapms{s} \varphi}_{L^p(\R^n\setminus\Omega)} \le C_{n,s,p} d^{-(\frac{n}{p'}-s)}\ \vrac{\varphi}_{L^1(K)}.
\]
\end{lemma}

\begin{proof}
Since convolution and multiplication are transformed into each other under Fourier transform and $(|\cdot|^s)^\wedge = c |\cdot|^{-s-n}$, for $x$ away from the support of $\varphi$ we have
$$\laps{s} \varphi(x) = c_{n,s}|\cdot|^{-n-s}\ast \varphi (x).$$
In particular
\[
 |\chi_{\R^n \backslash \Omega} \laps{s} \varphi| \leq \big(|\cdot|^{-n-s}\chi_{|\cdot| \geq \frac{d}{2}} \big)\ast \varphi.
\]
Now the first claim follows by Young's inequality:
\[
  \vrac{\chi_{\R^n \backslash \Omega} \laps{s} \varphi}_{L^p(\R^n)} \aleq \vrac{|\cdot|^{-n-s}\chi_{|\cdot| \geq \frac{d}{2}}}_{L^p(\R^n)}\ \vrac{\varphi}_{L^1(\R^n)}.
\]
The proof of the second claim is very similar.
\end{proof}

\begin{lemma}\label{disjointoperator}
Consider two functions $\theta_1,\theta_2 \in C_c^\infty(\R^n)$. Suppose that $\theta_1$ and $\theta_2$ have disjoint support, i.e. for some $d > 0$,
\begin{equation}\label{eq:do:disjointK}
 \dist (\operatorname{supp} \theta_1, \operatorname{supp} \theta_2 ) \ge d.
\end{equation}
For $s \in (0,n)$ let the operator $T$ be defined via 
\[
 Tf := \theta_1\, \lapms{s} (\theta_2 f), \quad f \in \Sw(\R^n).
\]
Then for any $t > 0$ the operator $ T \laps{t}$, originally defined on $\mathcal{S}(\R^n)$, extends to a linear bounded operator from $L^p(\R^n)$ to $L^q(\R^n)$ whenever $1+\frac{1}{q}- \frac{1}{p} \in [0,1]$,  with the estimate
\[
 \|T \laps{t}f \|_{L^q(\R^n)} \le C_{\theta_1,\theta_2,p,q,t} \|f\|_{L^p(\R^n)}
\]
\end{lemma}
\begin{proof}
First set
\[
 k(x,y) := |x-y|^{s-n}\, \theta_1(x)\, \theta_2(y).
\]
Notice that based on $\theta_1$ and $\theta_2$ and the disjoint support of the two functions \eqref{eq:do:disjointK} we can find $\theta_3 \in C^\infty_c(\R^n)$, $\theta_3 \equiv 0$ in the ball $B_{d/2}$ so that
\[
 k(x,y) = \theta_3(x-y) |x-y|^{s-n}\ \theta_1(x)\ \theta_2(y)\ 
\]
Note that 
\[
 \theta_3(\cdot) |\cdot|^{s-n} \in C^\infty(\R^n)
\]
In particular, $k(\cdot,y) \in C_c^\infty(\R^n)$ for any $y \in \R^n$ and $k(x,\cdot) \in C_c^\infty(\R^n)$ for any $x \in \R^n$. Morever, $\laps{t}_x k(x,\cdot) \in C_c^\infty(\R^n)$ for any $x \in \R^n$. Then for $f \in \Sw(\R^n)$
\[
\begin{split}
 T \laps{t} f(x) =&  \int_{\R^n} k(x,y) \laps{t}_y f(y)\, dy\\
=&\int_{\R^n} \laps{t}_y k(x,y)\,  f(y)\, dy,
\end{split}
\]
where we integrated by parts.

Setting
\[
 \tilde{k}(x,y) := \laps{t}_y k(x,y),
\]
and using that $(-\Delta)^\frac t2 \varphi(y)$ decays like $|y|^{-n-t}$ at infinity if $\varphi$ is compactly supported, we bound
$$\sup_{x \in \R^n} \|\tilde{k}(x,\cdot)\|_{L^{r}(\R^n)}<\infty,\quad \sup_{y \in \R^n}\|\tilde{k}(\cdot,y)\|_{L^{r}(\R^n)}<\infty$$
for every $r\in [1,\infty]$.
Then, by a straightforward adaption of Young's convolution inequality, if $1+\frac{1}{q} = \frac{1}{p} + \frac{1}{r}$ we get
\[\begin{split}
 \vrac{T \laps{t} f}_{L^q(\R^n)} &\leq \left(\sup_{x \in \R^n} \|\tilde{k}(x,\cdot)\|_{L^{r}(\R^n)}+\sup_{y \in \R^n}\|\tilde{k}(\cdot,y)\|_{L^{r}(\R^n)}\right)\, \vrac{f}_{L^p(\R^n)}\\
 & = C_{\theta_1,\theta_2,t,p,q}\, \vrac{f}_{L^p(\R^n)}. 
\end{split}
\]
\end{proof}

In some special cases we need to compute the constant in the Lemma above.
\begin{lemma}\label{la:freakshowestimate1}
For any $p \in [1,\infty]$, $q \in [1,\infty)$, any $\rho > 0$, $k \geq 2$, $s \in (0,n)$, we have the following estimate for any $f \in \Sw(\R^n)$,
\begin{equation}\label{eq:freakshowestimate1est1}
 \vrac{\lapms{s} (\theta_{A_{2^k\rho}} \laps{t} f)}_{L^p(B_\rho)} \aleq (2^k \rho)^{s-t-\frac{n}{q}}\ \rho^{\frac{n}{p}}\, \vrac{f}_{L^q(\R^n)}.
\end{equation}
Similarly, for any $g \in C_c^\infty(B_\rho)$,
\[
 \vrac{\laps{t} (\theta_{A_{2^k\rho}} \lapms{s} g)}_{L^{q'}(\R^n)}  \aleq (2^k \rho)^{s-t-\frac{n}{q}}\ \rho^{\frac{n}{p}} \vrac{g}_{L^{p'}(B_\rho)}.
\]

\end{lemma}
\begin{proof}
The second estimate follows from the first one by duality. Indeed
\[
 \begin{split}
\vrac{\laps{t} (\theta_{A_{2^k\rho}} \lapms{s} g)}_{L^{q'}(\R^n)}
& \aleq \sup_{f\in \Sw(\R^n),\|f\|_{L^q(\R^n)\le 1}}  \int_{\R^n} f\, \laps{t} (\theta_{A_{2^k\rho}} \lapms{s} g)\, dx \\
& =  \sup_{f\in \Sw(\R^n),\|f\|_{L^q(\R^n)}\le 1} \int_{\R^n} \lapms{s}(\theta_{A_{2^k\rho}} \laps{t}f)\, g\, dx\\
&\aleq  (2^k \rho)^{s-t-\frac{n}{q}}\ \rho^{\frac{n}{p}}\ \vrac{g}_{L^{p'}(\R^n)},
 \end{split}
\]
where we used integration by parts twice, cf. \eqref{eq:fracPI}, and \eqref{eq:freakshowestimate1est1}.
%, we thus have
%\[
%\begin{split}
%&\vrac{\laps{t} (\theta_{A_{2^k\rho}} \lapms{s} g)}_{q',\R^n} \aleq \|g\|_{p'}\ \vrac{\lapms{s} (\theta_{A_{2^k\rho}} \laps{t} f)}_{p,B_\rho} \\
%\overset{\eqref{eq:freakshowestimate1est1}}{\aleq} &\|g\|_{p'}\ (2^k \rho)^{s-t-\frac{n}{q}}\ \rho^{\frac{n}{p}}\ \|f\|_{q} \leq (2^k \rho)^{s-t-\frac{n}{q}}\ \rho^{\frac{n}{p}} \|g\|_{p'}
%\end{split}
%\]

The estimate \eqref{eq:freakshowestimate1est1} for $p\in [1,\infty)$ follows via H\"older's inequality from the case $p=\infty$ which we shall now prove. Up to scaling we can take $\rho=1$ and then \eqref{eq:freakshowestimate1est1} reduces to 
\begin{equation}\label{eq:freakshowestimate1est1b}
\vrac{\lapms{s} (\theta_{A_{2^k}} \laps{t} f)}_{L^\infty(B_1)} \aleq (2^k)^{s-t-\frac{n}{q}} \vrac{f}_{L^q(\R^n)}.
\end{equation}
For $k=2$ \eqref{eq:freakshowestimate1est1b} follows from Lemma~\ref{disjointoperator}:
\[
\vrac{\lapms{s} (\theta_{A_4} \laps{t} f)}_{L^\infty(B_1)} \le C_1 \vrac{f}_{L^q(\R^n)},
\]
with $C_1$ depending on $s,t,n,q$ and the chosen cut-off function $\theta_{B_1}$ (fixed in Definition \ref{def:cutoff}).
The case $k> 2$ follows from the case $k=2$ by scaling:
\[\begin{split}
\vrac{\lapms{s} (\theta_{A_{2^{k+2}}} \laps{t} f)}_{L^\infty(B_1)}&\le \vrac{\lapms{s} (\theta_{A_{2^{k+2}}} \laps{t} f)}_{L^\infty(B_{2^{k}})}\\
 &=(2^k)^{s-t}\vrac{\lapms{s} (\theta_{A_4} \laps{t} f(2^k\cdot))}_{L^\infty(B_1)}\\
 &\le C_1 (2^k)^{s-t}\|f(2^k\cdot)\|_{L^q(\R^n)}\\
 &= C_1 (2^k)^{s-t-\frac nq}\|f\|_{L^q(\R^n)}.
\end{split}\]
\end{proof}

Considering above $\theta_{A_{2^k \rho}} g$ instead of $\theta_{A_{2^k \rho}}$ we also have the following:
\begin{lemma}\label{la:freakshowestimate2}
For any $\rho > 0$, $p \in (1,\infty)$
\[
 \vrac{\lapms{s} (\theta_{A_{2^k\rho}} g (-\lap)^\frac n4 f)}_{L^p(B_\rho)} \aleq \max_{t \in [0,\frac{n}{2}]} \vrac{\laps{t}g}_{L^2(\R^n)} (2^k \rho)^{t-\frac{n}{2} + s - n}\rho^{\frac{n}{p}}\ \vrac{f}_{L^2(\R^n)},
\]
for any $f,g\in \Sw(\R^n)$.
\end{lemma}
\begin{proof}
By duality, the claim follows if we show for any $\varphi \in C_c^\infty(B_\rho)$
\begin{equation}\label{eq:freak2claim2}
  \vrac{(-\lap)^\frac{n}{4} (\theta_{A_{2^k\rho}} g(\lapms{s} \varphi))}_{L^2(\R^n)} \aleq \max_{t \in [0,\frac{n}{2}]} \vrac{\laps{t}g}_{L^2(\R^n)} (2^k \rho)^{t-\frac{n}{2} + s - m} \rho^{\frac{n}{p}}\ \vrac{\varphi}_{L^{p'}(\R^n)}.
\end{equation}
By the definition of the three-term-commutator $H_{\frac{n}{2}}$, H\"older inequality for a small $t > 0$, and the related estimates, see Theorem~\ref{th:bicomest}, 
\[
\begin{split}
\vrac{(-\lap)^\frac n4 (\theta_{A_{2^k\rho}} g(\lapms{s} \varphi))}_{L^2(\R^n)} 
 \aleq &\vrac{(-\lap)^\frac n4 g}_{L^2(\R^n)}\, \vrac{\theta_{A_{2^k\rho}} \lapms{s} \varphi}_{L^\infty(\R^n)} \\
 &+ \vrac{g}_{L^\frac{2n}{n-2t}(\R^n)}\, \vrac{(-\lap)^\frac n4(\theta_{A_{2^k\rho}} \lapms{s} \varphi)}_{L^\frac{n}{t}(\R^n)}\\ 
 &+ \| H_{\frac{n}{2}} (g,\theta_{A_{2^k\rho}} (\lapms{s} \varphi)) \|_{L^2(\R^n)}\\
\aleq &\vrac{(-\lap)^\frac n4 g}_{L^2(\R^n)}\ \vrac{\theta_{A_{2^k\rho}} \lapms{s} \varphi}_{L^\infty(\R^n)} \\
&+ \vrac{g}_{L^\frac{2n}{n-2t}(\R^n)}\, \vrac{(-\lap)^\frac n4(\theta_{A_{2^k\rho}} \lapms{s} \varphi)}_{L^\frac nt(\R^n)}\\ 
 &+ \| (-\lap)^\frac n4 g\|_{L^2(\R^n)}\ \vrac{(-\lap)^\frac n4(\theta_{A_{2^k\rho}} \lapms{s} \varphi)}_{L^2(\R^n)}\\
  \end{split}
 \]
By the Sobolev inequality,
\[
 \vrac{g}_{L^\frac{2n}{n-2t}(\R^n)} \aleq \vrac{\laps{t} g}_{L^2(\R^n)},
\]
and from Lemma~\ref{la:nonloc} with $\operatorname{supp} \varphi \subset B_\rho$
\[
 \vrac{\theta_{A_{2^k\rho}} \lapms{s} \varphi}_{L^\infty(\R^n)}  \aleq (2^k \rho)^{s-n} \vrac{\varphi}_{L^1(\R^n)} \aleq (2^k \rho)^{s-n} \rho^{\frac{n}{p'}}\ \| \varphi\|_{L^{p'}(\R^n)}.
\]
The remaining terms can be estimated with Lemma~\ref{la:freakshowestimate1}, and \eqref{eq:freak2claim2} follows.
\end{proof}

Let for $t > 0$ the three term commutator given as
\[
 H_t(a,b) := \laps{t}(ab) - b\laps{t}a -a\laps{t}b.
\]
A version similar to $H$ was first was introduced in \cite{DRSphere}. For subsequent similar results and extended arguments see also \cite{DRManifold,Sfracenergy},\cite[Lemma A.5]{BRS},\cite{DSpalphSphere}.
\begin{theorem}\label{th:bicomest}
Given $p \in (1,\infty)$, $t\ge 0$, $p_1, p_2 \in (1,\frac{n}{t}]$ satisfying
\[
 \frac{1}{p} = \frac{1}{p_1} + \frac{1}{p_2} - \frac{t}{n},
\]
it holds
%For any small $\eps \geq 0$, 
\[
 \vrac{H_t(a,b)}_{L^p(\R^n)} \aleq \vrac{\laps{t} a}_{L^{p_1}(\R^n)}\ \vrac{\laps{t} b}_{L^{p_2}(\R^n)},\quad \text{for }a,b\in \mathcal{S}(\R^n).
\]
\end{theorem}

\appendix
\section{Appendix}
\subsection{The Riesz transform and Riesz potential}
We define the Riesz potential of $u$ for $s \in (0,n)$ and $u \in \Sw(\R^n)$ 
\[
\lapms{s} u  := |\cdot|^{s-n} \ast u,
\]
By the density of the Schwartz class $\Sw(\R^n)$ in $L^p(\R^n)$, the Riesz potential $\lapms{s}$ can be extended to an operator mapping $L^p(\R^n)$ into $L^{\frac{np}{n-s}}(\R^n)$ whenever $p, \frac{np}{n-s} \in (1,\infty)$.
Up to a constant, the Riesz potential $\lapms{s}$ is the inverse of the fractional laplacian $\laps{s}$, in the sense that for a constant $c_{n,s} \in \R$
\[
\laps{s} \lapms{s} f =\lapms{s} \laps{s} f = c_{n,s} f \quad \forall f \in \Sw(\R^n).
\]
The Riesz transform $\Rz = (\Rz_1,\dots ,\Rz_n)$ is defined as 
\[
  \Rz u(x) := \int_{\R^n} \frac{x-y}{|x-y|^{n+1}}\, u(y)\, dy,\quad u\in \mathcal{S}(\R^n),
\]
and by density can be extended to a continuous operator from $L^p(\R^n)$ into itself:
\begin{equation}\label{Rieszprop-1}
\|\Rz u\|_{L^p(\R^n)}\le c_{p,n}\|u\|_{L^p(\R^n)}\quad \text{for }u\in L^p(\R^n).
\end{equation}
One crucial properties of the Riesz transform is the that
\begin{equation}\label{Rieszprop0}
\sum_{i=1}^n \Rz_i \Rz_i = c_n\, Id,
\end{equation}
and
\begin{equation}\label{Rieszprop}
 \Rz\laps{1} f = c_n\, \nabla f,\quad u\in \mathcal{S}(\R^n).
\end{equation}

We also recall the following property:

\begin{lemma}[``Integration by parts'']\label{la:pi}
For any $f \in L^p(\R^n)$, $g \in L^{p'}(\R^n)$, $p \in [1,\infty]$ so that $\Rz f \in L^p(\R^n)$, $\Rz g \in L^{p'}(\R^n)$ it holds
\begin{equation}\label{eq:RiezPI}
 \int_{\R^n} \Rz f\, g\, dx = -\int_{\R^n} f\, \Rz g\, dx. 
\end{equation}
For any $f \in L^p(\R^n)$, $g \in L^{p'}(\R^n)$, $p \in (1,\infty)$ so that $\laps{s} f \in L^{p}(\R^n)$, $\laps{s} g \in L^{p'}(\R^n)$,
\begin{equation}\label{eq:fracPI}
 \int_{\R^n} \laps{s}f\, g\, dx = \int_{\R^n} f\, \laps{s} g\, dx. 
\end{equation}
 \end{lemma}

Note that together \eqref{eq:RiezPI} and \eqref{eq:fracPI} imply the usual integration by parts formula
\[
 \int_{\R^n} \nabla f\, g\, dx = c\int_{\R^n} \Rz \laps{1} f\, g\, dx = -c\int_{\R^n}  f\, \Rz \laps{1}g\, dx = -\int_{\R^n}  f\, \nabla g
 \,dx.
\]
%This justifies (for one of the authors) the term ``integration by parts''.

\subsection{Lorentz spaces and Sobolev inequality}
\begin{definition}
For $1< p <\infty$ and $1\leq q \leq \infty$, we define the Lorentz space $L^{(p,q)}(\mathbb{R}^{n})$ as the space of measurable functions $f$ for which
$$\|f\|_{L^{(p,q)}}:=p^{1/q} \|\lambda |\{|f|>\lambda\}|^{1/p}\|_{L^{q}(\frac{d\lambda}{\lambda})}<\infty$$
\end{definition}
It is important to notice that $L^{(p,p)}=L^{p}$ and $L^{(p,q)}\subset L^{(p,r)}$ if $r\geq q$.

\begin{proposition}[Sobolev inequality]\label{pr:Sobolev}
Let $1<p<\frac{n}{\alpha}$ and $1\leq r \leq\infty$. If $f\in L^{(p,q)}(\mathbb{R}^{n})$, then $I_{\alpha}f \in L^{(q,r)}(\mathbb{R}^{n})$ for $q=\frac{np}{n-\alpha p}$. Moreover, there exists $C> $ such that
$$\|\lapms{\alpha}f\|_{L^{(p,r)}(\R^n)}\leq C\|f\|_{L^{(q,r)}(\R^n)}.$$
For $p=1$, $I_{\alpha}$ maps $L^1(\R^n)$ into $L^{(q,\infty)}(\R^n)$ for $q=\frac{n}{n-\alpha}$. 
For $p = \frac{n}{\alpha}$, $I_{\alpha}$ is bounded from $L^{(p,1)}$ into $L^\infty(\R^n)$.
\end{proposition}

From \cite[Corollary 6.16]{BS} we have
\begin{proposition}\label{pr:orliczsobolev}
$I_{\alpha}$ is a bounded linear operator from $L \log^r L(\R^n)$ to $L^{(\frac{n}{n-\alpha},\frac{1}{r})}(\R^n)$ whenever $r \leq 1$, $\alpha \in (0,n)$.
\end{proposition}

\subsection{Interior estimates}
The following are a few estimates which could be seen as $L^p$-theory for the fractional Laplacian. Since we only need interior estimates, the proofs are long, but elementary -- just relying on the definitions of Riesz potential, Riesz transform and fractional Laplacian. 
\begin{lemma}\label{la:sharmonicest}
Let $\Omega\subset\R^n$ be open. Then for any $h \in H^\frac{n}{2}(\R^n)$ satisfying 
\begin{equation}\label{eq:lan2heq0}
 \int_{\R^n} (-\lap)^\frac{n}{4} h\, (-\lap)^\frac{n}{4} \varphi\, dx = 0 \quad \forall \varphi \in C_c^\infty(\Omega),
\end{equation}
we have $h \in C^\infty_{\loc}(\Omega)$, and for any compact set $K\Subset\Omega$ and any $\ell \in \N_0$, $\alpha \in (0,1]$ we have
\[
 [\nabla^\ell h ]_{C^{0,\alpha}(K)} \leq C_{\ell,\alpha, K,\Omega}\ \vrac{(-\lap)^\frac n4 h}_{L^2(\R^n)}.
\]
\end{lemma}
\begin{proof}
The smoothness $h$, i.e. $h \in C^\infty_{\loc}(\Omega)$, follows via an approximation argument from the a priori estimates below.
Notice that \eqref{eq:lan2heq0} can be rewritten as

\begin{equation}\label{eq:lan2heq0bis}
 \int_{\R^n} \nabla h\cdot \nabla \laps{n-2} \varphi\,dx= 0, \quad \forall \varphi \in C_c^\infty(\Omega).
\end{equation}
Fix now $K \subset \subset K_1 \subset \subset K_2 \subset \subset \Omega$. For arbitary $\psi \in C_c^\infty(K_1)$ we have for $k \in \N_0$, 
\[
 \lap^k \psi = \laps{n-2} \lapms{n-2} \lap^k\psi.
\]
Thus, taking a cutoff-function $\eta_{K_2} \in C_c^\infty(\Omega)$, $\eta_{K_2} \equiv 1$ on $K_2$, 
\[
  \lap^k\psi = \laps{n-2} (\eta_{K_1}\lapms{n-2} \lap^k\psi) + \laps{n-2} ((1-\eta_{K_1}) \lapms{n-2} \lap^k\psi).
\]
Thus for any $\psi \in C_c^\infty(K_1)$, using \eqref{eq:lan2heq0bis} with $\varphi:=\eta_{K_2} I_{n-2}\Delta^k\psi$, we get
 \begin{align*}
\int_{\R^n}\nabla h \cdot \nabla  \lap^k \psi\, dx
= &\int_{\R^n} \nabla  h\cdot  \nabla \laps{n-2} ((1-\eta_{K_2}) \lapms{n-2} \lap^k\psi)\,dx \\
\leq & \vrac{\nabla h}_{L^n(\R^n)}  \vrac{\nabla \laps{n-2} ((1-\eta_{K_2}) \lapms{n-2} \lap^k\psi)}_{L^{n'}(\R^n)}\\
\leq & \vrac{\nabla h}_{L^n(\R^n)}  \vrac{\laps{n-1} ((1-\eta_{K_2}) \lapms{n-2} \lap^k\psi)}_{L^{n'}(\R^n)}\\
\leq & C_{K_1,K_2}\, \vrac{(-\lap)^\frac n4 h}_{L^2(\R^n)}\, \vrac{\psi}_{L^1(\R^n)}.
\end{align*}
The second-to-last step follows again from $\nabla = \Rz \laps{1}$ and because the Riesz transform $\Rz$ is bounded on $L^{n'}$. In the last step we used that the support of $1-\eta_{K_2}$ and $\psi$ are disjoint to apply Lemma~\ref{disjointoperator}, and Sobolev inequality. Classical regularity theory of elliptic PDE ensures that $h$ belongs to any Sobolev space $W^{\ell,p}_{\loc}(K_1)$ for any $\ell \in \N$, $p \in (1,\infty)$ together with the estimates
\[
 \vrac{h}_{W^{\ell,p}(K)} \leq C_{\ell,p,K,\Omega}\ (\vrac{(-\lap)^\frac n4 h}_{L^2(\R^n)} + \vrac{h}_{L^2(\R^n)}) \aleq  \|h\|_{H^{\frac{n}{2}}(\R^n)},
\]
and
\[
 \vrac{\nabla h}_{W^{\ell+1,p}(K)} \leq C_{\ell,p,K,\Omega}\ \vrac{(-\lap)^\frac n4 h}_{L^2(\R^n)}
\]
The latter implies via the Morrey-Sobolev imbedding that for any $\alpha \in (0,1)$, $l \in \N_0$
\[
 [\nabla^\ell h]_{C^{0,\alpha}(K)} \leq C_{\ell,\alpha, K,\Omega}\, \vrac{(-\lap)^\frac n4 h}_{L^2(\R^n)} .
 \]
\end{proof}

\begin{lemma}\label{la:LinftyRHSest}
Let $\Omega$ be an open set of $\R^n$. Then for any $h \in H^\frac n2(\R^n)$ satisfying 
\[
 \int_{\R^n} (-\lap)^\frac{n}{4} h\, (-\lap)^\frac n4 \varphi\, dx = \int_{\R^n} f \varphi\, dx \quad \forall \varphi \in C_c^\infty(\Omega),
\]
and for any $\ell \in \{0,1,\ldots,n-1\}$, $\alpha \in (0,1)$ we have on any compact $K \subset\subset \Omega$,
\[
[\nabla^\ell h]_{C^{0,\alpha}(K)} \leq C_{\ell,\alpha,\Omega,K}\, \vrac{(-\lap)^\frac n4 h}_{L^2(\R^n)} + C_{\ell,\Omega,K}\, \vrac{f}_{L^\infty(\Omega)}.
\]
\end{lemma}

\begin{proof}
The proof is very similar to the one of Lemma~\ref{la:sharmonicest}. Fix again $K \subset \subset K_1 \subset \subset K_2  \subset \subset \Omega$. 
% For arbitary $\psi \in C_c^\infty(K_2 )$ we now have
%  \begin{align*}
% \int \partial^\alpha h\ \partial^\alpha \lap^k \psi
% = &\int \partial^\alpha h\  \partial^\alpha \laps{n-2} ((1-\eta_{K_2 }) \lapms{n-2} \lap^k\psi) + \int f\ \eta_{K_2 } \lapms{n-2} \lap^k\psi\\
% \leq & C_{K_2,K_2 }\ \vrac{(-\lap)^{n/4} h}_{L^2}\ \vrac{\psi}_{L^1} + \|f\|_{\infty} \vrac{\lapms{n-2} \lap^k \psi}_{1,K_2 }\\
% \end{align*}
% For $0 \leq 2k \leq n-2$, (note that this implies that $2k \leq n-3$ for odd $n$), we have
% \[
%  \vrac{\lapms{n-2} \lap^k \psi}_{1,K_2 } \leq C_{K_2 } \vrac{\psi}_{L^1(\R^n)}\
% \]
% and thus elliptic regularity theory gives us that
% \[
%  h \in W^{n-1,p}_{loc}(K_2 ) \quad \mbox{for any $p \in (0,\infty)$},
% \]
% with the estimate
% \[
%  \vrac{\nabla h}_{W^{n-2,p}(K_2)} \leq \vrac{(-\lap)^{n/4} h}_{L^2(\R^n)} + \vrac{f}_{L^\infty}.
% \]
% In particular by Sobolev-Morrey embedding we have the claim for $l \in \{0,\ldots,n-3\}$.
% 
We use that the following equation (note that $n-1$ is even and thus $\laps{n-1}$ is the classical $(n-1)$-Laplacian),
\[
 \int_{\R^n} \laps{1} h\, \lap^{\frac{n-1}{2}} \varphi\, dx = \int_{\R^n} f \varphi\, dx \quad \forall \varphi \in C_c^\infty(\Omega).
\]
Elliptic theory implies $\laps{1} h \in W^{n-1,p}_{loc}(\Omega)$ for any $p \in (1,\infty)$, with the estimate
\begin{equation}\label{eq:lapshhwnm1p}
 \vrac{\laps{1}h}_{W^{n-1,p}(K_2 )}  \aleq \vrac{f}_{L^\infty(\Omega)} + \vrac{\laps{1} h}_{L^n(\R^n)} \aleq \vrac{f}_{L^\infty(\Omega)} + \vrac{(-\lap)^\frac n4 h}_{L^2(\R^n)}.
\end{equation}
Here again we used that $\laps{1} h \in L^n(\R^n)$ by Sobolev embedding. With this in mind, we can write $\nabla h$ in terms of the Riesz transform $\Rz$ and $(-\Delta)^\frac12$,
\begin{equation}\label{eq:partialalphahdecomp}
\nabla h = \Rz (\eta_{K_1} \laps{1} h) + \Rz ((1-\eta_{K_1}) \laps{1} h),
\end{equation}
where we have a cutoff function $\eta_{K_1} \in C_c^\infty(K_2 )$ and $\eta_{K_1} \equiv 1$ in $K_1$.
The first term on the right-hand side belongs to $W^{n-1,p}(\R^n)$ by \eqref{eq:lapshhwnm1p} and the boundedness of the Riesz transform, and we have
\[
 \vrac{\Rz (\eta_{K_1} \laps{1} h)}_{W^{n-1,p}(\R^n)} \aleq  \vrac{f}_{\infty} + \vrac{(-\lap)^\frac n4 h}_{L^2(\R^n)}.
\]
The second term on the right-hand side of \eqref{eq:partialalphahdecomp} is smooth in $K$, by the disjoint support of $\chi_K$ and $(1-\eta_{K_1})$. Indeed, by Lemma~\ref{la:nonloc} for any $\ell \geq 0$,
\begin{align*}
\vrac{\nabla^\ell \Rz ((1-\eta_{K_1}) \laps{1} h)}_{L^\infty(K)} \aleq C_{K,K_1}\ \vrac{\laps{1} h}_{L^n(\R^n)} \aleq C_{K,K_1}\ \vrac{(-\lap)^\frac n4 h}_{L^2(\R^n)}.
\end{align*}
Together, we have shown that for any $0 \leq \ell \leq n-1$, $p \in (1,\infty)$,
\[
 \vrac{\nabla h}_{W^{\ell,p}(K)} \aleq \vrac{f}_{L^\infty(\Omega)} + \vrac{(-\lap)^\frac n4 h}_{L^2(\R^n)}.
\]
Now the Sobolev-Morrey embedding gives the claim.
\end{proof}

\end{document}